\documentclass[11pt]{article}
\usepackage{amsfonts,amssymb,amsthm,amsmath,cite,bbm}

\newcommand{\N}{\mathbb N} 
\newcommand{\R}{\mathbb R} 
\newcommand{\Rn}{\R^n} 

\newcommand{\CV}{\operatorname{Conv}(\Rn)} 
\newcommand{\CVf}{\operatorname{Conv}(\Rn,\R)} 
\newcommand{\CVpa}{\operatorname{Conv_{p.a.}}(\Rn)} 

\newcommand{\K}{{\mathcal K}}
\newcommand{\Kn}{\K^n} 
\newcommand{\Ko}{\Kn_{o}} 
\newcommand{\KoOne}{\K_{o}^1} 
\newcommand{\Koin}{\Kn_{(o)}} 
\renewcommand{\P}{{\mathcal P}}
\newcommand{\Pn}{\P^n} 
\newcommand{\Po}{\Pn_{o}}
\newcommand{\Poin}{\Pn_{(o)}}
\newcommand{\Q}{{\mathcal Q}}
\newcommand{\Qn}{\Q^n} 

\newcommand{\cL}{{\mathcal L}} 
\newcommand{\cS}{{\mathcal S}}

\newcommand{\mx}{\mathbin{\vee}} 
\newcommand{\mn}{\mathbin{\wedge}} 

\newcommand{\Ind}{\mathrm{I}}

\renewcommand{\l}{\ell}
\newcommand{\e}{\varepsilon}

\renewcommand{\d}{\,\mathrm{d}}
\newcommand{\D}{\,\mathrm{D}}

\newcommand\SLn{\operatorname{SL}(n)}

\newcommand{\dom}{\operatorname{dom}}
\newcommand{\interior}{\operatorname{int}}
\newcommand{\cl}{\operatorname{cl}}
\newcommand{\conv}{\operatorname{conv}}
\newcommand{\pos}{\operatorname{pos}}
\newcommand{\epi}{\operatorname{epi}}
\newcommand{\reg}{\operatorname{reg}}
\newcommand{\argmax}{\operatorname{argmax}}

\newcommand{\elim}{\operatorname{epi-lim}}

\newcommand{\oZ}{\operatorname{Z}}
\newcommand{\oY}{\operatorname{Y}}

\newcommand{\eto}{\stackrel{epi}{\longrightarrow}}

\newtheorem{lemma}{Lemma}[section]
\newtheorem{theorem}[lemma]{Theorem}
\newtheorem*{theorem*}{Theorem}
\newtheorem{proposition}[lemma]{Proposition}

\newtheorem*{corollary*}{Corollary}

\theoremstyle{definition}

\theoremstyle{remark}
\newtheorem{remark}[lemma]{Remark}

\title{Volume, Polar Volume and Euler Characteristic for Convex Functions}
\author{Fabian Mussnig}
\date{}

\begin{document}

\maketitle

\begin{abstract}
Functional analogs of the Euler characteristic and volume together with a new analog of the polar volume are characterized as non-negative, continuous, $\SLn$ and translation invariant valuations on the space of finite, convex and coercive functions on $\Rn$.
\bigskip

{\noindent
2010 AMS subject classification: 26B25 (46A40, 52A20, 52A41, 52B45).}
\end{abstract}

\section{Introduction and Statement of the Main Result}
A map $\oZ$ defined on the subset $\cS$ of a lattice $(\cL,\vee, \wedge)$ and taking values in an abelian semigroup is called a \emph{valuation} if
\begin{equation*}
\label{eq:valuation}
\oZ(u\vee v)+\oZ(u\wedge v)=\oZ(u) +\oZ (v)
\end{equation*}
whenever $u,v, u\vee v, u\wedge v\in \cS$. Valuations defined on the set of convex bodies (compact convex sets), $\Kn$, in $\Rn$ have been studied since Dehn's solution of Hilbert's Third Problem in 1901. In this case, $\mx$ and $\mn$ denote union and intersection, respectively, and the set $\Kn$ is equipped with the topology induced by the Hausdorff metric. The first classification of valuations on convex bodies and a characterization of the Euler characteristic, $V_0$, and the $n$-dimensional volume (that is, the Lebesgue measure), $V_n$, was obtained by Blaschke \cite{blaschke}.

\begin{theorem}[Blaschke]
\label{thm:blaschke}
For $n\geq 2$, a map $\mu:\Kn\to\R$ is a continuous, $\SLn$ and translation invariant valuation if and only if there exist constants $c_0,c_1\in\R$ such that
$$\mu(K)=c_0 V_0(K)+c_1 V_n(K)$$
for every $K\in\Kn$.
\end{theorem}

Here, a valuation $\mu:\Kn\to\R$ is called \emph{$\SLn$ invariant} if $\mu(\phi K) = \mu(K)$ for all $\phi\in\SLn$ and $K\in\Kn$, where $\phi K =\{\phi x\,:\,x\in K\}$. Moreover, $\mu$ is said to be \emph{translation invariant} if $\mu(K+x)=\mu(K)$ for all $x\in\Rn$, with $K+x = \{y+x\,:\,y\in K\}$. See also \cite{Alesker99,Alesker01,bernig_fu,haberl,ludwig_reitzner_annals,  haberl_parapatits_moments,schuster_wannerer_mink_and_generalized, li_leng_orlicz,li_ma,abardia_wannerer,hlyz_acta} for some recent results on valuations on convex bodies and \cite{hadwiger,klain_rota} for more information on the classical theory.\par
More recently, valuations on function spaces have been introduced and studied. Here, $u\vee v$ denotes the pointwise maximum of $u$ and $v$ and $u\wedge v$ the pointwise minimum of two functions $u,v\in\cS$, where $\cS$ is a space of real-valued functions. The first results for valuations on Sobolev spaces were obtained by Ludwig \cite{ludwig_sobolev,ludwig_fisher} and Ma \cite{ma}. For $L^p$ spaces complete classifications for valuations intertwining the $\SLn$ were established in \cite{tsang_mink_val_on_lp, tsang_val_on_lp, ludwig_covariance}. Bobkov, Colesanti and Fragal{\`a} \cite{bobkov_colesanti_fragala} as well as Milman together with Rotem \cite{milman_rotem} extended intrinsic volumes to the space of quasi-concave functions (see also \cite{colesanti_fragala, klartag_milman}). A classification of rigid motion invariant valuations on quasi-concave functions was established by Colesanti and Lombardi \cite{colesanti_lombardi} and for definable functions such a result was previously established by Baryshnikov, Ghrist and Wright \cite{baryshnikov_ghrist_wright}. For further results, see also \cite{Alesker_convex,colesanti_lombardi_parapatits,colesanti_ludwig_mussnig_mink,ludwig_function, kone,wang_semi_val, ober_minkowski, tradacete_villanueva_adv,villanueva}.\par
For convex functions, a characterization of functional analogs of the Euler characteristic and volume was established in \cite{colesanti_ludwig_mussnig}. Let $\CV$ denote the space of all convex functions $u:\Rn\to(-\infty,+\infty]$ that are proper, lower semicontinuous and coercive. Here a function $u$ is \emph{proper} if it is not identically $+\infty$ and it is called \emph{coercive} if 
\begin{equation*}
\label{eq:coercive}
\lim_{\lvert x\rvert \to+ \infty} u(x)=+\infty.
\end{equation*}
Furthermore, the space $\CV$ is equipped with the topology associated to epi-convergence (see also Section \ref{se:convex_functions}).\par
We say that a map $\oZ$, defined on a space of real-valued functions $\cS$ on $\Rn$, is \emph{translation invariant} if $\oZ(u\circ \tau^{-1})=\oZ(u)$ for every $u\in\cS$ and translation $\tau:\Rn\to \Rn$. Moreover, $\oZ$ is called \emph{$\SLn$ invariant} if $\oZ(u\circ\phi^{-1}) = \oZ(u)$ for every $u\in\cS$ and $\phi\in\SLn$.

\begin{theorem}[\!\!\cite{colesanti_ludwig_mussnig}]
For $n\geq 2$, a map $\oZ:\CV\to[0,\infty)$ is a continuous, $\SLn$ and translation invariant valuation if and only if there exist continuous functions $\zeta_0,\zeta_1: \R \to [0,\infty)$ where $\zeta_1$ has finite moment of order $n-1$ such that 
\begin{equation*}
\oZ(u) = \zeta_0\big(\min\nolimits_{x\in\Rn}u(x)\big) + \int_{\dom u} \zeta_1\big(u(x)\big) \d x
\end{equation*}
for every $u\in\CV$.
\end{theorem}

Here, a function $\zeta:\R \to [0,\infty)$ has \emph{finite moment of order $n-1$} if $\int_0^{+\infty} t^{n-1} \zeta(t)\d t<+\infty$ and $\dom u$ denotes the \emph{domain} of $u$, that is, $\dom u=\{x\in\R^n: u(x)<+\infty\}$. Furthermore, for functions $u\in\CV$, the minimum is attained and hence finite.\par
The proof of this theorem uses the following classification of continuous and $\SLn$ invariant valuations on $\Ko$, the space of convex bodies which contain the origin. A functional $\mu:\Ko\to\R$ is a continuous and $\SLn$ invariant valuation if and only if there exist constants $c_0,c_1\in\R$ such that
\begin{equation*}
\mu(K)=c_0 V_0(K)+ c_n V_n(K)
\end{equation*}
for every $K\in\Ko$ (see, for example, \cite[Corollary 1.2]{ludwig_reitzner}). However, if one restricts to the class $\Koin$ of convex bodies that contain the origin in their interiors, an additional $\SLn$ invariant valuation appears. Therefore, let
$$K^*=\{x\in\Rn\,:\,x\cdot y \leq 1,\,\forall y\in K\}$$
denote the \emph{polar set} of $K\in\Kn$. If $K$ contains the origin in its interior, then also $K^*$ is an element of $\Koin$ and hence bounded. In this case, $K^*$ is also said to be the \emph{polar body} of $K$. We now define $V_n^*(K)=V_n(K^*)$ as the \emph{polar volume} of $K\in\Koin$, which is an important quantity in convex geometric analysis. For example the famous but still unsolved Mahler conjecture states that the minimum of $V_n^*(K)$ among all origin-symmetric bodies $K\in\Koin$ with $V_n(K)=1$ is attained by a hypercube. The first characterization of the polar volume was obtained by Ludwig in \cite{ludwig_val_poly_origin_int}. More recently, a long conjectured classification similar to Theorem~\ref{thm:blaschke} was obtained by Haberl and Parapatits.

\begin{theorem}[\!\!\cite{haberl_parapatits_centro}]
\label{thm:haberl_parapatits_centro}
For $n\geq 2$, a map $\mu:\Koin\to\R$ is a continuous and $\SLn$ invariant valuation if and only if there exist constants $c_0,c_1,c_2\in\R$ such that
\begin{equation*}
\mu(K)=c_0V_0(K)+c_1 V_n(K)+c_2 V_n^*(K),
\end{equation*}
for every $K\in\Koin$.
\end{theorem}
In order to establish an analog of this result for convex functions, let
\begin{equation*}
u^*(x)=\sup\nolimits_{y\in\Rn} \big(x\cdot y - u(y)\big),\qquad x\in\Rn
\end{equation*}
denote the \emph{convex conjugate} of a function $u:\Rn\to[-\infty,\infty]$, where $x\cdot y$ denotes the inner product of $x,y\in\Rn$. If $u$ is proper and does not attain $-\infty$, then the function $u^*:\Rn\to(-\infty,+\infty]$ is always convex, proper and lower semicontinuous. Furthermore, convex conjugation is a continuous operation and is compatible with $\SLn$ transforms (for details see Section \ref{se:convex_functions}). Let $\CVf$ denote the space of all convex, coercive functions $u:\Rn\to\R$. We will prove the following result.

\begin{theorem*}
For $n\geq 2$, a map $\oZ:\CVf\to[0,\infty)$ is a continuous, $\SLn$ and translation invariant valuation if and only if there exist continuous functions $\zeta_0,\zeta_1,\zeta_2: \R \to [0,\infty)$ where $\zeta_1$ has finite moment of order $n-1$ and $\zeta_2(t)=0$ for all $t\geq T$ with some $T\in\R$ such that 
\begin{equation}
\label{eq:z_u}
\oZ(u) = \zeta_0\big(\min\nolimits_{x\in\Rn}u(x)\big) + \int_{\Rn} \zeta_1\big(u(x)\big) \d x + \int_{\dom u^*} \zeta_2 \big(\nabla u^*(x)\cdot x - u^*(x) \big) \d x
\end{equation}
for every $u\in\CVf$.
\end{theorem*}

Here, $\nabla u$ denotes the \emph{gradient} of a function $u$ defined on $\Rn$. Note, that by Rademacher's theorem (see for example \cite[Theorem 3.1.6]{federer}) a convex function is differentiable almost everywhere on the interior of its domain.\par
\begin{remark}
If (\ref{eq:z_u}) is evaluated for a so-called cone function $\l_K\in\CVf$ with $K\in\Koin$, that is a function which is defined via its \emph{sublevel sets}
$$\{\l_K \leq t \} := \{x\in\Rn\,:\, \l_K(x)\leq t\}=t K,$$
for $t\geq 0$ and $\{\l_K \leq t\}=\emptyset$ for $t<0$, then a linear combination of $V_0(K), V_n(K)$ and $V_n^*(K)$ is obtained.
\end{remark}

\begin{remark}
For a function $u\in\CVf\cap C^2(\R^n)$, the new term in (\ref{eq:z_u}) can be rewritten as
$$\int_{\Rn} \zeta_2(u(x)) \det(\!\D^2u(x)) \d x,$$
where $\D^2 u(x)$ is the \emph{Hessian matrix} of $u$ and $\det(\!\D^2u(x))$ denotes its \emph{determinant}. This is also a special case of the so-called \emph{Hessian valuations} that were introduced in \cite{colesanti_ludwig_mussnig_hess}.
\end{remark}

In addition to the main result, we will study functional analogs of further $\SLn$ invariant valuations in Section~\ref{subse:further_vals}.

\section{Valuations on Convex Bodies}
In this section we consider some basic results about valuations on convex bodies and constructions on polytopes. Let $\Pn\subset \Kn$ denote the set of convex polytopes and let $\Po$ and $\Poin$ denote the subsets of polytopes that contain the origin and polytopes that contain the origin in their interiors, respectively.  For general references, we refer to the books of Gruber \cite{gruber} and Schneider \cite{schneider}.\par
A real-valued valuation $\mu$ defined on some subset $\Qn\subseteq \Kn$ is said to be \emph{homogeneous of degree $i\in\R$} if $\mu(\lambda K) = \lambda^i \mu(K)$ for every $\lambda >0$ and $K\in\Qn$. Furthermore, for $i\in\N$, $\mu$ is said to be \emph{$i$-simple} if $\mu(K)=0$ for every $K\in\Qn$ with $\dim K < i$. For example, the Euler characteristic $V_0$ is homogeneous of degree $0$ and the $n$-dimensional volume $V_n$ is homogeneous of degree $n$ and $n$-simple. Moreover, we have for any $K\in\Koin$ and $\lambda >0$
$$V_n^*(\lambda K) = V_n((\lambda K)^*) = V_n(\lambda^{-1} K^*) = \lambda^{-n} V_n^*(K),$$
which shows that the polar volume is homogeneous of degree $-n$.\par
The next result is due to \cite{mcmullen,mcmullen77} (see also \cite[Corollary 6.3.2., Theorem 6.3.5]{schneider}).

\begin{theorem}[McMullen decomposition]
\label{thm:mcmullen_decomp}
Let $\mu:\Kn\to\R$ be a translation invariant, continuous valuation. There exist continuous, translation invariant valuations $\mu_0,\ldots,\mu_n$ on $\Kn$ such that $\mu_i$ is homogeneous of degree $i$ and $i$-simple for $0\leq i \leq n$ and
$$\mu(\lambda K) = \sum_{i=0}^n \lambda^i \mu_i(K)$$
for every $K\in\Kn$ and $\lambda\geq 0$. In particular, $\mu = \mu_0 + \cdots + \mu_n$.
\end{theorem}

For the proof of our main result we will need some constructions on polytopes that will be used as level sets of certain convex functions. In the following, let $\{e_1,\ldots,e_n\}$ denote the standard basis of $\Rn$ and let $T^n$ denote the standard simplex in $\Rn$, that is
$$T^n=\conv\{0,e_1,\ldots,e_n\},$$
where $\conv$ denotes the \emph{convex hull}. Furthermore, set $\bar{e}:=(1,\ldots,1)^t$. For $\delta>0$ we set $T_\delta:=(1+2\delta)T^n - \delta \bar{e} \in \Pn$ or equivalently
$$T_\delta = \conv\left\{\begin{pmatrix}
   -\delta\\
   -\delta\\
   \vdots\\
   -\delta
\end{pmatrix}, \begin{pmatrix}
1+\delta\\
-\delta\\
\vdots\\
-\delta
\end{pmatrix}, \cdots, \begin{pmatrix}
-\delta\\
\vdots\\
-\delta\\
1+\delta
\end{pmatrix} \right\}.$$
Note, that if $n=2$ or $n\geq 3$ and $0<\delta<\tfrac{1}{n-2}$, then $T_\delta$ contains the origin in its interior.

\begin{lemma}
\label{le:p_t}
Let $0<\delta<\tfrac{1}{n-2}$ if $n\geq 3$ and $0<\delta < 1$ if $n=2$. For $b>0$, $0<\rho< 1$ and $t\geq b$ let $x_\delta = (1+\delta,-\delta,\cdots,-\delta)^t$ and let $P_{\delta,\rho}^{b,t}:=t T_\delta \cap \{x_1 \leq b(1+\delta)+ \rho(t-b) \}$. It holds that
\begin{eqnarray}
\label{eq:pt_cup}P_{\delta,\rho}^{b,t} \cup \big((t-b)T_\delta + b x_\delta\big) &=& t T_\delta\\
\label{eq:pt_cap}P_{\delta,\rho}^{b,t} \cap \big((t-b)T_\delta + b x_\delta\big) &=& (t-b)(T_\delta\cap \{x_1 \leq \rho \}) + b x_\delta
\end{eqnarray}
for every $t\geq b$. Furthermore
$$V_n^*(T_\delta \cap \{x_1\leq \rho\})= V_n^*(T_\delta)+\frac{1}{n!\delta^{n-2}} \frac{1+\delta}{\delta(1-(n-2)\delta)}\left(\frac{1}{\rho}-\frac{1}{1+\delta}\right).$$
\end{lemma}
\begin{proof}
We will show identities (\ref{eq:pt_cup}) and (\ref{eq:pt_cap}) by adding the vector $\delta t \bar{e}$ first. Note, that by definition
\begin{equation}
\label{eq:t_t_delta}
s T_\delta+\delta s\bar{e} = s c_\delta T^n,
\end{equation}
for any $s>0$, where $c_\delta = (1+2\delta)$. Therefore
$$P_{\delta,\rho}^{b,t} + \delta t \bar{e} = t c_\delta T^n \cap \{x_1 \leq b(1+\delta) + \delta t +\rho (t-b)\}.$$
Furthermore,
\begin{align*}b x_\delta + \delta t \bar{e} &=(b(1+\delta) + \delta t,\delta (t-b),\ldots,\delta(t-b))^t\\
&= \delta (t-b) \bar{e} + b c_\delta  e_1,
\end{align*}
which shows that
\begin{align*}
(t-b)T_\delta + b x_\delta + \delta t \bar{e} &= (t-b) T_\delta + \delta (t-b) \bar{e} + b c_\delta e_1\\
&= (t-b) c_\delta T^n + b c_\delta e_1.
\end{align*}
Hence, using (\ref{eq:t_t_delta}) again, equation (\ref{eq:pt_cup}) is equivalent to
$$\big(t c_\delta T^n \cap \{x_1 \leq b(1+\delta) + \delta t + \rho(t-b) \}\big) \cup \big((t-b) c_\delta  T^n + b c_\delta e_1 \big) = t c_\delta T^n,$$
which follows from the intercept theorem and the fact that
$$b(1+\delta)+\delta t + \rho (t-b) \geq b (1+2\delta) = b c_\delta.$$
Moreover,
\begin{align*}
(t-b)(T_\delta \cap \{x_1 \leq \rho \}) + b x_\delta &+ \delta t \bar{e}\\
&= \big( (t-b)T_\delta \cap \{x_1 \leq \rho (t-b) \}\big) + \delta (t-b) \bar{e} + b c_\delta e_1\\
&=\big((t-b) c_\delta T^n + b c_\delta e_1 \big) \cap \{x_1 \leq b(1+\delta) + \delta t + \rho (t-b) \}.
\end{align*}
This shows that (\ref{eq:pt_cap}) is equivalent to
\begin{align*}
\big(t c_\delta T^n \cap \{x_1 \leq b(1+\delta) + \delta t + &\rho (t-b) \}\big) \cap \big((t-b) c_\delta  T^n + b c_\delta e_1 \big)\\&= \big((t-b) c_\delta T^n + b c_\delta e_1 \big) \cap \{x_1 \leq b(1+\delta) + \delta t + \rho (t-b)\},
\end{align*}
which is easy to see.\par
In order to show the second statement, note that
$$T_\delta = \{x\cdot (1,0,\ldots,0)^t\leq -\delta \}\cap  \ldots \cap \{x\cdot (0,\ldots,0,1)^t\leq -\delta \} \cap \{x\cdot (1,\ldots,1)^t \leq 1-(n-2)\delta\}.$$
Hence,
$$T_\delta^* =
\conv \left\{
\begin{pmatrix} -1/\delta \\ 0 \\ \vdots \\ 0 \end{pmatrix},
\begin{pmatrix} 0 \\ -1/\delta \\ \vdots \\ 0 \end{pmatrix},
\cdots,
\begin{pmatrix} 0 \\ 0 \\ \vdots \\ -1/\delta \end{pmatrix},
\begin{pmatrix} 1/(1-(n-2)\delta) \\ 1/(1-(n-2)\delta) \\ \vdots \\ 1/(1-(n-2)\delta) \end{pmatrix}
\right\}.
$$
Furthermore it is easy to see that $(T_\delta \cap \{x_1 \leq \rho\})^* = \conv \{T_\delta^*, (1/\rho,0,\ldots,0)^t\}$. Hence,
$$V_n^*(T_\delta \cap \{x_1 \leq \rho\}) = V_n^*(T_\delta) + V_n(K_{\delta,\rho})$$
with
$$ K_{\delta,\rho} =
\conv \left\{
\begin{pmatrix} 1/\rho \\ 0 \\ \vdots \\ 0 \end{pmatrix},
\begin{pmatrix} 0 \\ -1/\delta \\ \vdots \\ 0 \end{pmatrix},
\cdots,
\begin{pmatrix} 0 \\ 0 \\ \vdots \\ -1/\delta \end{pmatrix},
\begin{pmatrix} 1/(1-(n-2)\delta) \\ 1/(1-(n-2)\delta) \\ \vdots \\ 1/(1-(n-2)\delta) \end{pmatrix}
\right\}.
$$
We use Laplace's formula (along the last column) to calculate $V_n(K_{\delta,\rho})=V_n(K_{\delta,\rho}+(0,\ldots,0,1/\delta)^t)$ which is given by
$$
V_n(K_{\delta,\rho})=\frac{1}{n!}
\left|\det
\begin{pmatrix}
1/\rho & 0 & \cdots & 0 & 1/(1-(n-2)\delta) \\ 0 & -1/\delta & \cdots & 0 & 1/(1-(n-2)\delta) \\ \vdots & \vdots & \ddots & \vdots & \vdots \\ 0 & 0 & 0 & -1/\delta & 1/(1-(n-2)\delta)\\ 1/\delta & 1/\delta & 1/\delta & 1/\delta & 1/(1-(n-2)\delta)+1/\delta
\end{pmatrix}\right|.
$$
This gives
\begin{align*}
n! V_n(K_{\delta,\rho}) &= \big|(-1)^{n-1} \tfrac{1}{1-(n-2)\delta} (-1)^{n-1} (\tfrac{-1}{\delta})^{n-1}+(-1)^n \tfrac{1}{1-(n-2)\delta} \tfrac{1}{\rho} (-1)^n (\tfrac{-1}{\delta})^{n-2}\\
&\,\,\,\quad + \cdots + (\tfrac{1}{1-(n-2)\delta}+\tfrac{1}{\delta})\tfrac{1}{\rho} (\tfrac{-1}{\delta})^{n-2}\big|\\
&= \big| \tfrac{1}{1-(n-2)\delta}(\tfrac{-1}{\delta})^{n-1} + \tfrac{1}{\rho} (\tfrac{-1}{\delta})^{n-2} (\tfrac{n-1}{1-(n-2)\delta}+\tfrac{1}{\delta}) \big|\\
&=\big|(\tfrac{-1}{\delta})^{n-2} \tfrac{1+\delta}{\delta(1-(n-2)\delta)} (\tfrac{1}{\rho}-\tfrac{1}{1+\delta})\big|\\
&=\tfrac{1}{\delta^{n-2}} \tfrac{1+\delta}{\delta(1-(n-2)\delta)} (\tfrac{1}{\rho}-\tfrac{1}{1+\delta})
\end{align*}
which completes the proof.
\end{proof}

In the following, we write $B^n=\{x\in\Rn\,:\,|x|\leq 1\}$ for the unit ball in $\Rn$, $Q^n=[-1,1]^n$ for the centered standard cube in $\Rn$ and
$$C^n:=\conv\{\pm e_1,\ldots, \pm e_n\}=(Q^n)^*\in\Poin$$
for the corresponding cross-polytope. We will need the following result.

\begin{lemma}
\label{le:vol_conv_poly}
Let $c_i\geq 0$ for $1\leq i\leq n$ and let $K:=\conv\{0,c_1 e_1,\ldots, c_n e_n\}$. For $\delta>0$, we have
$$V_n(\conv(\delta C^n \cup K)) = \frac{1}{n!}\prod_{i=1}^n (\max\{c_i,\delta\}+\delta).$$
\end{lemma}
\begin{proof}
This follows easily if one considers that for $a_i,b_i\geq 0$, $1\leq i\leq n$ one has
$$V_n(\conv\{a_1 e_1,-b_1 e_1,\ldots, a_n e_n, -b_n e_n \}) = \frac{1}{n!}\prod_{i=1}^n (a_i + b_i),$$
as well as
$$\conv(\delta C^n \cup K) = \conv\{\max\{c_1,\delta\} e_1,-\delta e_1,\ldots, \max\{c_n,\delta\} e_n,-\delta e_n \}.$$
\end{proof}

\section{Convex Functions}
\label{se:convex_functions}
We collect some results on convex functions. For basic references we refer to the books of Rockafellar \& Wets \cite{rockafellar_wets} and Schneider \cite{schneider}. To every convex function $u:\Rn\to( -\infty,+\infty]$ one can assign several convex sets. The (effective) domain of $u$, $\dom u = \{x \in \Rn\,:\,u(x) < +\infty\}$, is a convex subset of $\Rn$ and the \emph{epigraph} of $u$, $\epi u = \{(x, y) \in \Rn\times\R\,:\,u(x) \leq y\}$,
is a convex subset of $\Rn\times\R$. For $t\in \R$, the sublevel set,
$$\{u \leq t\}=\{x\in\Rn\,:\,u(x)\leq t\},$$
is convex. If $u$ is coercive, then its sublevel sets are bounded and if $u$ is lower semicontinuous the sublevel sets are closed. Hence, if $u\in\CV$, the sets $\{u\leq t\}$ are elements of $\Kn$ for every $t\geq \min_{x\in\Rn} u(x)$. In particular, this minimum is attained and finite and the space $\CV$ can be seen as a functional analog of $\Kn$. Note, that for $u,v\in\CV$ and $t\in\R$
$$\{u \mn v \leq t\} = \{u \leq t\} \cup \{v \leq t\}\quad \text{and} \quad \{u \mx v \leq t\} = \{u\leq t\} \cap \{v \leq t\},$$
where for $u \mn v \in \CV$ all occurring sublevel sets are either empty or in $\Kn$.\par
The space $\CV$ is equipped with the topology associated to epi-convergence. A sequence $u_k:\Rn\to (-\infty,+\infty]$ is said to be \emph{epi-convergent} to $u:\Rn\to (-\infty,+\infty]$ if for all $x\in\Rn$ the
following conditions hold:
\begin{itemize}
\item[(i)] For every sequence $x_k$ that converges to $x$, $u(x)\leq \liminf_{k\to+\infty} u_k(x_k)$.
\item[(ii)] There exists a sequence $x_k$ that converges to $x$ such that $u(x)=\lim_{k\to+\infty} u_k(x_k)$.
\end{itemize}
In other words, $u$ is an optimal common asymptotic lower bound of the sequence $u_k$. For epi-convergent sequences $u_k$ with limit function $u$ we also write $u=\elim_{k\to+\infty} u_k$ and $u_k \eto u$.
\begin{remark}
Another name for epi-convergence is $\Gamma$-convergence.
\end{remark}

Epi-convergence is strongly connected to Hausdorff convergence of sublevel sets. In the following result (see \cite[Lemma 5]{colesanti_ludwig_mussnig} and \cite[Theorem 3.1]{beer_rockafellar_wets}) we say that $\{u_k \leq t\} \to \emptyset$ as $k\to+\infty$ if there exists $k_0\in\N$ such that $\{u_k\leq t\}=\emptyset$ for all $k\geq k_0$.
\begin{lemma}
\label{le:hd_conv_lvl_sets}
Let $u_k,u\in\CV$. If $u_k \eto u$, then $\{u_k\leq t\} \to \{u\leq t\}$ as $k\to+\infty$ for every $t\in\R$ with $t\neq \min_{x\in\Rn} u(x)$. Furthermore, if for every $t\in\Rn$ there exists a sequence $t_k \to t$ such that $\{u_k\leq t_k\} \to \{u\leq t\}$, then $u_k \eto u$.
\end{lemma}

Another consequence of epi-convergence is due to \cite[Theorem 7.17]{rockafellar_wets}

\begin{theorem}
\label{thm:g_comp}
Let $u_k,u:\Rn\to(-\infty,+\infty]$ be convex functions. If $u_k$ epi-converges to $u$, then $u$ is convex. Moreover, if $\dom u$ has non-empty interior, then $u_k$ converges uniformly to $u$ on every compact set that does not contain a boundary point of $\dom u$. 
\end{theorem}

\begin{remark}
By Theorem \ref{thm:g_comp}, epi-convergence on $\CVf$ is equivalent to local uniform convergence and compact convergence. Furthermore, by \cite[Example 5.13]{dal_maso}, epi-convergence on this function space is also equivalent to pointwise convergence.
\end{remark}

A consequence of coerciveness is the so called \emph{cone property}, which was established in \cite[Lemma 2.5]{colesanti_fragala}.

\begin{lemma}
\label{le:cone}
For $u\in\CV$ there exist constants $a,b \in \R$ with $a >0$ such that
\begin{equation*}
u(x)>a|x|+b
\end{equation*}
for every $x\in\Rn$.
\end{lemma}

Furthermore, a \emph{uniform cone property} was established in \cite[Lemma 8]{colesanti_ludwig_mussnig}.

\begin{lemma}
\label{le:un_cone}
For $u_k, u\in\CV$ such that $u_k \eto u$, there exist constants $a,b\in\R$ with $a>0$ such that
\begin{equation*}
u_k(x)>a|x|+b\quad \text{and}\quad u(x)>a|x|+b
\end{equation*}
for every $k\in\N$ and $x\in\Rn$.
\end{lemma}

Recall, that for a convex function $u:\Rn\to[-\infty,+\infty]$, the convex conjugate $u^*$ is defined by
$$u^*(x)=\sup\nolimits_{y\in\Rn} \big(x\cdot y - u(y) \big),$$
for every $x\in\Rn$. The next result can be found in \cite[Theorem 1.6.13]{schneider}.
\begin{lemma}
\label{le:u**}
If $u:\Rn\to(-\infty,+\infty]$ is a proper, lower semicontinuous, convex function, then also $u^*$ is a proper, lower semicontinuous, convex function and $u^{**} = u$.
\end{lemma}

In the following we say that a function $u:\Rn\to (-\infty,+\infty]$ is \emph{super-coercive} if
$$\lim_{|x|\to+\infty} \frac{u(x)}{|x|} = +\infty.$$
Furthermore, let $\interior A$ denote the \emph{interior} of a set $A\subset\Rn$. As the next result shows, certain properties of a convex function correspond to certain other properties of its conjugate function, see for example \cite[Theorem 11.8]{rockafellar_wets}. 
\begin{lemma}
\label{le:props_of_u*}
For a proper, lower semicontinuous, convex function $u:\Rn\to(-\infty,+\infty]$, the following hold true:
\begin{itemize}
	\item $u$ is coercive if and only if $0\in \interior \dom u^*$.
	\item $u$ is super-coercive if and only if $\dom u^* = \Rn$.
\end{itemize}
\end{lemma}

Convex conjugation is also compatible with the valuation property, see for example \cite[Lemma 3.4, Proposition 3.5]{colesanti_ludwig_mussnig_hess} 

\begin{lemma}
\label{le:conjugate_is_a_val}
Let $u,v$ be proper, lower semicontinuous, convex functions. If $u\mn v$ is convex, then so is $u^*\mn v^*$. Furthermore,
$$(u\mn v)^* = u^* \mx v^*\qquad \text{and}\qquad (u\mx v)^* = u^* \mn v^*.$$
\end{lemma}

The next result, which is due to Wijsman, shows that convex conjugation is a continuous operation (see \cite[Theorem 11.34]{rockafellar_wets}).

\begin{theorem}
\label{th:wijsman}
If $u_k,u:\Rn\to(-\infty,+\infty]$ are closed, proper and convex, then
\begin{equation*}
u_k \eto u \qquad \text{if and only if} \qquad u_k^*\eto u^*.
\end{equation*}
\end{theorem}

For a convex, lower semicontinuous function $u:\Rn\to(-\infty,+\infty]$ and $x\in\Rn$, a vector $y\in\Rn$ is said to be a \emph{subgradient of $u$ at $x$} if
$$u(z)\geq u(x) + (z-x)\cdot y,$$
for all $z\in\Rn$. The (possibly empty) set of all subgradients at $x$ is called the \emph{subdifferential of $u$ at $x$} and denoted by $\partial u(x)$. In particular, if $u$ is differentiable at $x$, then $\partial u(x) = \{\nabla u(x)\}$. The next result uses subdifferentials to establish a connection between a convex function and its conjugate (see, for example, \cite[Theorem 23.5]{rockafellar}).

\begin{lemma}
\label{le:subgr_conj}
For a proper, lower semicontinuous, convex function $u:\Rn\to(-\infty,+\infty]$ and $x,y\in\Rn$, the following are equivalent:
\begin{itemize}
	\item $y\in\partial u(x)$,
	\item $x\in\partial u^*(y)$,
	\item $x\cdot y = u(x) +u^*(y)$,
	\item $x\in\argmax_{z\in\Rn} (y\cdot z - u(z))$,
	\item $y\in\argmax_{z\in\Rn} (x\cdot z - u^*(z))$.
\end{itemize}
\end{lemma}
Here, $\argmax_{z\in V} f(z)$ denotes the points in the set $V$ at which the function values of $f$ are maximized on $V$.\par
For $K\in\Kn$ we consider the convex indicator function $I_K\in\CV$, which is defined as
$$
I_K(x)=\begin{cases}0,\quad &x\in K\\ +\infty,\quad &x\notin K. \end{cases}
$$
Furthermore, for $K\in\Ko$ we will consider the cone function $\l_K\in\CV$, which is defined via
$$\epi \l_K = \pos (K \times \{1\}),$$
where $\pos$ denotes the \emph{positive hull} or equivalently
\begin{equation}
\label{eq:lvl_set_lk}
\{\l_K\leq t\} = t\,
\end{equation}
for every $t\geq 0$. Note, that if $K\in\Koin$, then $\l_K\in \CVf$. Furthermore, for every $K\in\Koin$ we have $\l_K=h(K^*,\cdot)$ where $h(K^*,\cdot)$ denotes the \emph{support function} of $K^*$. For convex bodies $L\in \Kn$ the support function can be defined as
$h(L,x)=\max\nolimits_{y\in L} y\cdot x$
for every $x\in\Rn$.
\par
The following relations for $u\in\CV$, $K\in\Ko$, $L\in\Kn$, $t\in\R$, $x\in\Rn$, translations $\tau_y(x)=x+y$ and $\phi\in\SLn$ are easy to see:
\begin{align}
\label{eq:properties_conj}
\begin{split}
(u+t)^*(x) \; &= \; u^*(x)-t\\
(\l_K)^*(x) \; &= \; \Ind_{K^*}(x)\\
\Ind_L^*(x) \; &= \; h(L,x)\\
(u\circ \tau_y^{-1})^*(x) \; &= \; u^*(x)+y\cdot x\\
(u\circ \phi^{-1})^*(x) \; &= \; (u^*\circ\phi^t)(x)
\end{split}
\end{align}

Next, for $\delta>0$ and $u\in\CV$ we define the regularization $\reg_{\delta} u$ as
\begin{equation}
\label{eq:def_reg_delta}
\reg_{\delta} u = (u^* + \Ind_{\delta^{-1} Q^n})^*.
\end{equation}

\begin{lemma}
\label{le:lip_reg}
For $u,u_j,v\in\CV$ with $u\wedge v\in\CV$, $K\in\Ko$ and $\delta>0$ we have the following properties.
\begin{enumerate}
	\item $\reg_{\delta} u \in \CVf$.
	\item $\reg_{\delta} (u\wedge v) = \reg_{\delta} u \wedge \reg_{\delta} v$ and $\reg_{\delta} (u\vee v) = \reg_{\delta} u \vee \reg_{\delta} v$.
	\item $\reg_{\delta} u \eto u$ as $\delta \to 0$.
	\item If $u_j\to u$, then $\reg_{\delta} u_j \eto \reg_{\delta} u$ for sufficiently small $\delta>0$.
	\item $\reg_{\delta} (u+t) = (\reg_{\delta} u) + t$ for every $t\in\R$.
	\item $\reg_{\delta} (u \circ \tau^{-1}) = (\reg_{\delta} u)\circ \tau^{-1}$ for every translation $\tau$ on $\Rn$.
	\item There exists a unique $K_{\delta}\in\Koin$ such that $\reg_{\delta} \l_K = \l_{K_{\delta}}$. In particular, $K_{\delta} = \conv(K \cup \delta C^n)$. Furthermore, $K \subseteq K_{\delta}$ with equality iff $\delta C^n \subseteq K$. In particular, $K_{\delta}\to K$ in the Hausdorff metric as $\delta\to 0$.
\end{enumerate}
\end{lemma}
\begin{proof}
The proofs of 1, 2, 3 and 4 follow along similar lines as the proofs of corresponding results for the well-known Lipschitz regularization, see for example \cite[Section 4]{colesanti_ludwig_mussnig_hess}. Since for any proper, lower semicontinuous, convex function $u$ and $t\in\R$ we have $(u+t)^*=u^*-t$, it follows for $u\in\CV$ and $\delta>0$ that
\begin{align*}
\reg_{\delta}(u+t)=(u^*-t+\Ind_{\delta^{-1} Q^n})^* = (u^*+\Ind_{\delta^{-1} Q^n})^*+t= (\reg_{\delta} u) + t.
\end{align*}
Similarly, one shows covariance with respect to translations.\par
Next, let $K\in\Ko$ and observe that
$$\reg_{\delta} \l_K = (\Ind_{K^*} + \Ind_{\delta^{-1} Q^n})^* = (\Ind_{K^* \cap \delta^{-1} Q^n})^* = \l_{(K^* \cap \delta^{-1} Q^n)^*}.$$
Since $K^* \cap \delta^{-1} Q^n$ is an element of $\Koin$, we have $\reg_{\delta} \l_K = \l_{K_{\delta}}$ with
$$K_{\delta} := (K^* \cap \delta^{-1} Q^n)^* = \conv(K \cup \delta C^n) \in \Koin,$$
see for example \cite[Theorem 1.6.3]{schneider}. Furthermore, since polarity is order reversing, it is easy to see that $K\subseteq K_{\delta}$. Finally, $\delta C^n\subseteq K$ is equivalent to $K=K_{\delta}$.
\end{proof}

\begin{remark}
Another way to define $\reg_{\delta} u$ would be
$$\epi \reg_{\delta} u := \epi u + \epi \l_{\delta C^n},$$
where $+$ denotes the \emph{Minkowski addition} (that is adding each element of one set with each element of the other set). For details see \cite[Section 1.6]{schneider}.
\end{remark}

To proof Lemma~\ref{le:reduction} we will also need another class of convex functions. We call a function $u\in\CV$ \emph{piecewise affine}, if there exist finitely many $n$-dimensional convex polyhedra
$Q_1, \ldots , Q_m$ with pairwise disjoint interiors such that $\bigcup_{i=1}^m Q_i = \Rn$ and the restriction of $u$ to each $Q_i$ is an affine function. We will denote the set of piecewise affine convex functions by $\CVpa$.\par
Since epi-convergence on $\CVf$ is equivalent to pointwise convergence, the following result is easy to see (see also \cite[Lemma 11]{colesanti_ludwig_mussnig}).

\begin{lemma}
\label{le:dense}
$\CVpa$ is a dense subset of $\CVf$.
\end{lemma}

\section{Valuations on Convex Functions}
In this section we will consider several $\SLn$ invariant valuations on spaces of convex functions, most of which can be interpreted as functional analogs of well known operators on $\Kn$.\par
In the following we say that a real-valued valuation $\oZ$, defined on a subset $\mathcal{S}$ of $\CV$, is \emph{homogeneous of degree $p\in\R$} if $\oZ(u_\lambda) = \lambda^p \oZ(u)$ for every $u\in\mathcal{S}$ and $\lambda>0$, where $u_\lambda(x):=u(\frac x\lambda)$. Note, that $(u_\lambda)^*=(u^*)_{\lambda^{-1}}$.
\subsection{Minimum and Integral}
The following operator can be seen as a functional analog of the Euler characteristic.

\begin{lemma}[\!\!\cite{colesanti_ludwig_mussnig}, Lemma 12]
\label{le:min_is_a_val}
For a continuous function $\zeta:\R\to\R$ the map
$$
u\mapsto \zeta(\min\nolimits_{x\in\Rn} u(x))
$$
defines a continuous, $\SLn$ and translation invariant valuation on $\CV$ that is homogeneous of degree $0$.
\end{lemma}

The properties of the next operator are similar to those of the volume operator of convex bodies.

\begin{lemma}[\!\!\cite{colesanti_ludwig_mussnig}, Lemma 16]
\label{le:vol_is_a_val}
For a continuous function $\zeta:\R\to[0,\infty)$ with finite moment of order $n-1$, the map
$$
u\mapsto \int_{\dom u} \zeta(u(x)) \d x
$$
defines a non-negative, continuous, $\SLn$ and translation invariant valuation on $\CV$ that is homogeneous of degree $n$.
\end{lemma}

\subsection{Polar Volumes}
By Lemma~\ref{le:props_of_u*} it is easy to see that
\begin{align*}
\CV^*:&=\{u^*\,:\,u\in\CV\}\\
&=\{u:\Rn\to(-\infty,\infty]\,:\, u \text{ is proper, l.s.c., convex}, 0\in\interior\dom u\}.
\end{align*}
Now for a valuation $\oZ$ on $\CV$, the map
$$u\mapsto \oZ^*(u)=\oZ(u^*)$$
defines a valuation on $\CV^*$ since by Lemma~\ref{le:conjugate_is_a_val}
\begin{align*}
\oZ^*(u\mn v) + \oZ^*(u\mx v) &= \oZ((u\mn v)^*) + \oZ((u\mx v)^*)\\
&= \oZ(u^* \mx v^*) + \oZ(u^* \mn v^*)\\
&= \oZ(u^*) + \oZ(v^*)\\
&= \oZ^*(u) + \oZ^*(v)
\end{align*}
for every $u,v\in\CV^*$ such that $u\mn v\in\CV^*$. Furthermore, if $\oZ$ is continuous and $\SLn$ invariant, then by Theorem~\ref{th:wijsman} and (\ref{eq:properties_conj}) the operator $\oZ^*$ is also continuous and $\SLn$ invariant. Moreover, if $\oZ$ is translation invariant, then $\oZ^*$ is \emph{linear invariant}, that is
$$\oZ^*(u+l) = \oZ^*(u)$$
for every $u\in\CV^*$ and $l:\Rn\to\Rn$ of the form $l(x)=y\cdot x$ with some $y\in\Rn$. Lastly, if $\oZ$ is homogeneous of degree $p\in\R$, then
$$\oZ^*(u_\lambda)=\oZ((u_\lambda)^*) = \oZ((u^*)_{\lambda^{-1}}) = \lambda^{-p} \oZ(u^*) = \lambda^{-p} \oZ^*(u),$$
which shows that $\oZ^*$ is homogeneous of degree $-p$. Hence, by Lemma~\ref{le:vol_is_a_val} we have the following result.

\begin{lemma}
\label{le:polar_vol1}
For a continuous function $\zeta:\R\to[0,\infty)$ with finite moment of order $n-1$, the map
\begin{equation}
\label{eq:polar_vol1}
u\mapsto \int_{\dom u^*} \zeta(u^*(x)) \d x
\end{equation}
defines a non-negative, continuous, $\SLn$ and linear invariant valuation on $\CV^*$ that is homogeneous of degree $-n$.
\end{lemma}

\begin{remark}
If (\ref{eq:polar_vol1}) is evaluated for functions of the type $\l_K$ or $\Ind_K$ with $K\in\Koin$, a multiple of $V_n^*(K)$ is obtained. However, if we evaluate for the support function $h(K,\cdot)$ of a convex body $K\in\Kn$, we obtain
$$\int_{\dom h^*(K,\cdot)} \zeta(h^*(K,x)) \d x = \int_{ K } \zeta(\Ind_K(x))\d x = \zeta(0) V_n(K).$$
Note, that for every $y\in\Rn$ we have $h(K+y,x) = h(K,x) + y\cdot x$ for $x\in\Rn$. In particular, the valuation defined in (\ref{eq:polar_vol1}) is not translation invariant anymore.
\end{remark}

In order to obtain a translation invariant analog of the polar volume on $\CVf$ we need the following result.
\begin{proposition}[\!\!\cite{colesanti_ludwig_mussnig_hess}, Section 10.4 and Theorem 11.1]
\label{pr:hess_is_cont_val}
For a continuous function $\zeta:\R\times\Rn\to\R$ with compact support, the map
$$u\mapsto \int_{\dom u^*} \zeta(\nabla u^*(x)\cdot x - u^*(x), x) \d x$$
defines a continuous and translation invariant valuation on $\CV$.
\end{proposition}

\begin{lemma}
\label{le:polar_vol_is_a_val}
For a continuous function $\zeta:\R\to\R$ such that $\zeta(t)=0$ for all $t\geq T$ with some $T\in\R$, the map
\begin{equation}
\label{eq:val_polar_volume}
u \mapsto \int_{\dom u^*} \zeta(\nabla u^*(x)\cdot x - u^*(x)) \d x
\end{equation}
defines a continuous, $\SLn$ and translation invariant valuation on $\CVf$ that is homogeneous of degree $-n$.
\end{lemma}
\begin{proof}
Let $u_k,u\in\CVf$ such that $u_k \eto u$. By Lemma~\ref{le:un_cone}, there exist $b\in\R$ and $R>0$ such that $u_k(x),u(x)\geq b$ for every $x\in\Rn$ and $u_k(x),u(x)\geq T$ for every $x\in \Rn$ with $|x|\geq R$ and every $k\in\N$. By Theorem~\ref{thm:g_comp} the functions $u_k$ converge uniformly to $u$ on $B_R^n:=R\cdot B^n$. Furthermore, as convex functions they are Lipschitz continuous on this set and therefore there exists $C>0$ such that $|y|\leq C$ for every $y\in \partial u_k(x)\cup \partial u(x)$ with $x\in B_R^n$ and $k\in\N$. Hence, $u(x),u_k(x)\geq T$ for every $x\in\Rn$ such that there exists $y\in \partial u(x) \cup u_k(x)$ with $|y|> C$. Lemma~\ref{le:subgr_conj} now shows that for every pair $x,y\in\Rn$ such that $y\in\partial u(x)$ and $|y|>C$ we have
$$x\cdot y - u^*(y) = u(x) \geq T$$
and similarly if $y\in\partial u_k(x)$ with $|y|> C$
$$x\cdot y - u_k^*(y) = u_k(x) \geq T,$$
for every $k\in\N$. Using Lemma~\ref{le:subgr_conj} again and considering that convex functions are differentiable a.e. this gives
$$\nabla u^*(y)\cdot y - u^*(y) \geq T\qquad \text{and}\qquad \nabla u_k^*(y)\cdot y - u_k^*(y) \geq T$$
for a.e. $y\in\Rn$ with $|y|> C$ and for every $k\in\N$.\par
Next, let $\xi:\Rn\to \R$ be a smooth function with compact support such that $\xi\equiv 1$ on $B_C^n$ and let $\eta:\R\to\R$ be a smooth function with compact support such that $\eta\equiv 1$ on $[b,T]$. Define $\bar{\zeta}\in C(\R\times \Rn)$ as $\bar{\zeta}(t,y):=\zeta(t,y)\eta(t)\xi(y)$ for every $t\in\R$ and $y\in\Rn$. Using Proposition~\ref{pr:hess_is_cont_val} we now have
\begin{multline*}\int_{\dom u_k^*} \zeta(\nabla u_k^*(x)\cdot x - u_k^*(x)) \d x = \int_{\dom u_k^*} \bar{\zeta}(\nabla u_k^*(x)\cdot x - u_k^*(x),x) \d x\\
\stackrel{k \to \infty}{\longrightarrow} \int_{\dom u^*} \bar{\zeta}(\nabla u^*(x)\cdot x - u^*(x),x) \d x = \int_{\dom u^*} \zeta(\nabla u^*(x)\cdot x - u^*(x)) \d x,
\end{multline*}
which shows that (\ref{eq:val_polar_volume}) defines a continuous, translation invariant valuation on $\CVf$.\par
It remains to show $\SLn$ invariance and homogeneity. Therefore, let $\phi\in\SLn, \lambda>0$ and observe that
\begin{align*}
\int_{\dom (u_\lambda \circ \phi^{-1})^*} \zeta( \nabla (u_\lambda \circ \phi^{-1})^*(x) &\cdot x - (u_\lambda\circ \phi^{-1})^*(x)) \d x\\
&= \int_{\dom u^*\circ \lambda \phi^t} \zeta (\nabla (u^* \circ \lambda\phi^t)(x)\cdot x - u^*(\lambda \phi^t x))\d x\\
&= \int_{(\lambda \phi^{t})^{-1} \dom u^*} \zeta(\nabla u^*(\lambda \phi^t x) \cdot \lambda \phi^t x - u^*(\lambda \phi^t x)) \d x \\
&= \lambda^{-n} \int_{\dom u^*} \zeta(\nabla u^*(x)\cdot x - u^*(x)) \d x,
\end{align*}
which concludes the proof.
\end{proof}

\begin{remark}
Lemma~\ref{le:polar_vol_is_a_val} shows that there exist non-trivial, continuous and translation invariant valuations on $\CVf$ that are homogeneous of degree $-n$. Hence, a direct analog of Theorem~\ref{thm:mcmullen_decomp} for valuations on $\CVf$ is not true.
\end{remark}

\begin{remark}
It is easy to see that (\ref{eq:val_polar_volume}) does not extend to $\CV$. In order to see that, let $u=\Ind_K+t$ with $K\in\Koin$ and $t\in\R$ such that $\zeta(t)\neq 0$. We have $u^*=\l_{K^*}-t$ and furthermore $\nabla \l_{K^*}(x)\cdot x = \l_{K^*}(x)$ for a.e. $x\in\Rn$. Thus,
$$\int_{\dom u^*} \zeta(\nabla u^*(x)\cdot x - u^*(x))\d x = \int_{\Rn} \zeta(t) \d x$$
which is not finite.
\end{remark}

\subsection{Further $\SLn$ Invariant Valuations}
\label{subse:further_vals}
We briefly discuss further $\SLn$ invariant valuations on spaces of convex functions. It is easy to see that $K\mapsto V_0(K^*)$ defines an $\SLn$ invariant, continuous valuation on $\Koin$. Obviously, in this case $V_0(K^*)=1=V_0(K)$ and such an operator does not explicitly appear in Theorem~\ref{thm:haberl_parapatits_centro}. For a convex function $u$ on $\Rn$ however, the values $\zeta(\min_{x\in\Rn}u(x))$ and $\zeta(\min_{x\in\Rn} u^*(x))$ do not coincide in general. Note, that
$$u^*(0)=\sup\nolimits_{y\in\Rn} (0-u(y)) = -\inf\nolimits_{y\in\Rn} u(y),$$
and similarly $u(0)=u^{**}(0)=-\inf_{y\in\Rn} u^*(y)$. Hence, $u^*$ is bounded from below and attains its minimum if and only if $0\in \dom u$. Thus, using Theorem~\ref{th:wijsman} we have the following result dual to Lemma~\ref{le:min_is_a_val}

\begin{lemma}
For a continuous function $\zeta:\R\to\R$ the map
$$u\mapsto \zeta(\min\nolimits_{x\in\Rn} u^*(x))=\zeta(-u(0))$$
defines a continuous, $\SLn$ and linear invariant valuation on
$$\{u:\Rn\to(-\infty,\infty]\,:\, u \text{ is proper, l.s.c., convex}, 0\in \dom u\}$$
that is homogeneous of degree $0$. 
\end{lemma}

If one removes the assumption of translation invariance in Theorem~\ref{thm:blaschke}, the additional valuation
$$K\mapsto V_n(\conv\{0,K\})$$
appears, see for example \cite[Corollary 2.3]{ludwig_reitzner}. As a functional analog of $\conv\{0,K\}$ on $\CV$ we associate with a function $u\in\CV$ the function $u_0\in\CV$ which is defined via
$$u_0:=(u^*\mx 0)^*.$$
It is not hard to see that this can also be written as
$$\epi u_0 = \conv\{0\times[0,\infty),\epi u\}.$$

\begin{lemma}
For a continuous function $\zeta:\R\to[0,\infty)$ with finite moment of order $n-1$, the map
\begin{equation}
\label{eq:vol_u0}
u\mapsto \int_{\dom u_0} \zeta(u_0(x)) \d x
\end{equation}
defines a non-negative, continuous and $\SLn$ invariant valuation on $\CV$.
\end{lemma}
\begin{proof}
Note, that by Lemma~\ref{le:props_of_u*} the function $u_0$ is coercive if and only if $0\in\interior\dom u_0^*$, which is true since $u$ is coercive and therefore $0 \in\interior \dom u^*$. Hence, by Lemma~\ref{le:vol_is_a_val} the map (\ref{eq:vol_u0}) is well defined. Furthermore, by  Theorem~\ref{th:wijsman} and (\ref{eq:properties_conj}) it is easy to see that this operator is continuous and $\SLn$ invariant. It remains to show the valuation property. Therefore, let $u,v\in\CV$ such that $u\mn v\in\CV$. Since
\begin{align*}
(u^*\mx 0) \mx (v^* \mx 0) &= (u^* \mx v^*) \mx 0\\
(u^* \mx 0) \mn (v^* \mx 0) &= (u^* \mn v^*) \mx 0
\end{align*}
it follows from Lemma~\ref{le:conjugate_is_a_val} and Lemma~\ref{le:vol_is_a_val} that (\ref{eq:vol_u0}) defines a valuation.
\end{proof}

Similar to Lemma~\ref{le:polar_vol1} we immediately obtain the following result.
\begin{lemma}
For a continuous function $\zeta:\R\to[0,\infty)$ with finite moment of order $n-1$, the map
$$u\mapsto \int_{\dom (u^*)_0} \zeta((u^*)_0 (x)) \d x$$
defines a non-negative, continuous and $\SLn$ invariant valuation on $\CV^*$.
\end{lemma}

We want to close this section with a result dual to Lemma~\ref{le:polar_vol_is_a_val}. Note, that the space dual to $\CVf$ is
\begin{align*}
\CVf^*&= \{u^*\,:\,u\in\CVf\}\\
&=\{u:\Rn\to(-\infty,\infty]\,:\,u \text{ is proper, l.s.c., convex, super-coercive}, 0\in\interior\dom u\}.
\end{align*}
\begin{lemma}
For a continuous function $\zeta:\R\to\R$ such that $\zeta(t)=0$ for all $t\geq T$ with some $T\in\R$, the map
$$u\mapsto \int_{\dom u} \zeta(\nabla u(x)\cdot x - u(x))\d x$$
defines a continuous, $\SLn$ and linear invariant valuation on $\CVf^*$ that is homogeneous of degree $n$.
\end{lemma}

\section{Classification of Valuations}
\subsection{General Considerations}
The following result is a variant of \cite[Lemma 17]{colesanti_ludwig_mussnig} and is based on a principle that was introduced in \cite[Lemma 8]{ludwig_sobolev}. As there was an error in the induction step of the original proof, we give a new corrected proof here. Furthermore, the author is most grateful to Jin Li for pointing out the mistake.

\begin{lemma}
\label{le:reduction}
Let $\langle A,+ \rangle$ be a topological abelian semigroup with cancellation law and let $\oZ_1,\oZ_2:\CVf\to\langle A,+ \rangle$ be continuous, translation invariant valuations. If $\oZ_1(\l_P+t)= \oZ_2(\l_P+t)$ for every $P\in\Poin$ and $t\in\R$, then $\oZ_1\equiv\oZ_2$ on $\CVf$.
\end{lemma}
\begin{proof}
By Lemma~\ref{le:dense} and the continuity of $\oZ_1$ and $\oZ_2$, it is enough to show that $\oZ_1$ and $\oZ_2$ coincide on $\CVpa$. Hence, w.l.o.g. let $u\in\CVpa$ and set $U=\epi u \subset \Rn\times \R$. Since $u$ does not attain the value $+\infty$, none of the facet hyperplanes of $U$ (i.e. the hyperplanes in $\R^{n+1}$ that have an $n$-dimensional intersection with the boundary of $U$) is parallel to the $x_{n+1}$-axis. We call $U$ \emph{singular} if $U$ has $n$ facet hyperplanes whose intersection contains a line parallel to the coordinate hyperplane $\{x_{n+1}=0\}$. By continuity it is enough to restrict to the cases where $U$ is not singular.\par
We will use induction on the number $m$ of vertices of $U$. If $m=1$, then $U$ has just one vertex $p_0=(x_0,t_0)$ with $x_0\in\Rn$ and $t_0\in\R$ and therefore $U$ is the translate of a polyhedral cone. Denote by $Q\in\Pn$ the projection onto the first $n$ coordinates of $U\cap \{x_{n+1}=t_0+1\}$ and let $P=Q-x_0$. Since none of the facet hyperplanes of $U$ is parallel to the $x_{n+1}$-axis and $U$ is polyhedral, the set $P$ is a polytope that contains the origin in its interior and $u$ is a translate of $\l_P+t_0$. As $\oZ_1$ and $\oZ_2$ are translation invariant, we have $\oZ_1(u)=\oZ_2(u)$.\par
Now let $U$ have $m>1$ vertices and assume that $\oZ_1$ and $\oZ_2$ coincide on every piecewise affine convex function whose epigraph has most $m-1$ vertices. Since $U$ is not singular, there exists a unique vertex $p_0=(x_0,t_0)$ of $U$ with $x_0\in\Rn$ and $t_0\in\R$ such that $t_0$ is minimal. Let $H_1,\ldots,H_j$ be the facet hyperplanes of $U$ that contain $p_0$ and define $U_0$ as the intersection of the corresponding half-spaces. By the properties of $U$ and the choice of $p_0$ the set $U_0$ is the epigraph of a function $u_0\in\CVpa$ such that $u_0$ is a translate of $\l_{P_0}+t_0$ where $P_0\in\Poin$ is such that $P_0+x_0$ is the projection onto the first $n$ coordinates of $U_0\cap \{x_{n+1}=t_0+1\}$.\par
Next, let $p_1\subset \R^{n+1}$ be a vertex of $U$ with second smallest $x_{n+1}$ coordinate. Since $u$ is a convex piecewise affine function there exists $\e>0$ such that $p_1+\epi \l_{[-\e,\e]^n}\subset U \subset U_0$. Let $F_1,\ldots,F_j$ denote the facets of $U_0$ and set $\overline{F}_i = F_i \backslash (F_i \cap U)$ for $1\leq i \leq j$. We now define the polyhedron $U_1$ as
$$U_1 = \cl \big(\conv\{\overline{F}_1,\ldots,\overline{F}_j, p_1 + \epi \l_{[-\e,\e]^n}\}\big)\subset U_0,$$
where $\cl$ denotes the closure of a set. By definition, $U_1$ is the epigraph of a polyhedral convex function $u_1\in\CV$. Since $p_1 + \epi \l_{[-\e,\e]^n}\subset U_1$ the function $u_1$ does not attain the value $+\infty$ and therefore $u_1 \in \CVpa$. Furthermore, $U_1$ and $U\cap U_1$ each have at most $m-1$ vertices since every vertex of $U_1$ is also vertex of $U$ but $p_0$ is not a vertex of $U_1$. Hence, by the induction assumption $\oZ_1$ and $\oZ_2$ coincide on $u_1$ and $u \mx u_1$. Moreover, since $U \cup U_1 = U_0$ we have $u \mn u_1 = u_0$ and therefore by the valuation property 
\begin{align*}
\oZ_1(u)+\oZ_1(u_1)&=\oZ_1(u \mx u_1) + \oZ_1(u \mn u_1)\\
&=\oZ_1(u \mx u_1) + \oZ_1(u_0)\\
&= \oZ_2(u\mx u_1) + \oZ_2(u_0)\\
&=\oZ_2(u \mx u_1) + \oZ_2(u \mn u_1)= \oZ_2(u)+\oZ_2(u_1),
\end{align*}
which completes the proof.
\end{proof}

\subsection{Considerations on $\SLn$ Invariant Valuations}
In the following let $n\geq 2$.

\begin{lemma}
\label{le:growth_functions}
If $\oZ:\CVf\to[0,+\infty)$ is a continuous and $\SLn$ invariant valuation, then there exist continuous functions $\psi_0,\psi_1,\psi_2:\R\to[0,+\infty)$ such that
$$\oZ(\l_K+t) = \psi_0(t) + \psi_1(t) V_n(K) + \psi_2(t) V_n^*(K)$$
for every $K\in\Koin$ and $t\in\R$.
\end{lemma}
\begin{proof}
For $t\in\R$, define $\mu_t:\Koin \to \R$ as
$$\mu_t(K)=\oZ(\l_K+t).$$
Since for $K,L\in\Koin$ such that $K\cup L \in\Koin$ one has
$$\l_{K\cup L} = \l_K \mn \l_L\qquad \text{and} \qquad \l_{K\cap L} = \l_K \mx \l_L,$$
the map $\mu_t$ defines a valuation on $\Koin$ for every $t\in\R$. Furthermore, by (\ref{eq:lvl_set_lk}) and Lemma~\ref{le:hd_conv_lvl_sets} it is easy to see that $\mu_t$ is continuous and $\SLn$ invariant. Hence, by Theorem~\ref{thm:haberl_parapatits_centro}, there exist constants $c_{0,t}, c_{1,t}, c_{2,t}\in\R$ such that
$$\oZ(\l_K+t)=\mu_t(K)=c_{0,t}+c_{1,t}V_n(K)+c_{2,t} V_n^*(K),$$
for every $K\in\Koin$. This defines functions $\psi_0(t)=c_{0,t}$, $\psi_1(t)=c_{1,t}$ and $\psi_2(t) = c_{2,t}$. Fix $K\in\Koin$. For every $\lambda>0$ and $t\in\R$ we have
\begin{equation}
\label{eq:oZ_homo}
\oZ(\l_{\lambda K}+t) = \psi_0(t) + \lambda^n \psi_1(t) V_n(K) + \lambda^{-n} \psi_2(t) V_n^*(K).
\end{equation}
Considering, that $t\mapsto \oZ(\l_{\lambda K}+t)$ is continuous and taking different values for $\lambda$ and linear combinations of (\ref{eq:oZ_homo}), it is easy to see, that $\psi_0$, $\psi_1$ and $\psi_2$ must be continuous. For example, subtracting (\ref{eq:oZ_homo}) with $\lambda=1$ from the general case gives
$$\oZ(\l_{\lambda K} +t) - \oZ(\l_K+t) = (\lambda^n-1)\psi_1(t) V_n(K) + (\lambda^{-n}-1)\psi_2(t)V_n^*(K).$$
Taking again different values for $\lambda$ in the equation above and considering that $V_n(K)\neq 0$ for every $K\in\Koin$, one can see that $\psi_1$ is continuous. Similarly, one can see that $\psi_0$ and $\psi_2$ are continuous. Furthermore, by homogeneity, it is easy to see that $\psi_0$, $\psi_1$ and $\psi_2$ are non-negative.
\end{proof}

For a continuous and $\SLn$ invariant valuation $\oZ:\CVf\to\R$, we call the functions $\psi_0$, $\psi_1$ and $\psi_2$ the \emph{growth functions} of $\oZ$. By Lemma~\ref{le:reduction}, we have the following result.

\begin{lemma}
\label{le:val_det_by_gf}
Every continuous, $\SLn$ and translation invariant valuation $\oZ:\CVf \to \R$ is uniquely determined by its growth functions.
\end{lemma}

In order to classify valuations on $\CVf$, we need to determine the properties of their growth functions.

\begin{lemma}
\label{le:growth_functions_cond}
If $\oZ:\CVf\to[0,+\infty)$ is a continuous, $\SLn$ and translation invariant valuation, then its growth function $\psi_1$ satisfies $\lim_{t\to+\infty} \psi_1(t)=0$ and there exists $T\in\R$ such that $\psi_2(t)=0$ for all $t\geq T$.
\end{lemma}
\begin{proof}
Fix $0<\delta<\tfrac{1}{n-2}$ if $n\geq 3$ and $0<\delta < 1$ if $n=2$ as well as $0<\rho<1$. For $b>0$ and $t\geq b$ let $T_\delta$, $P_{\delta,\rho}^{b,t}$ and $x_\delta$ be as in Lemma~\ref{le:p_t} and define $u_{\delta,\rho}^{b}\in\CVf$ as 
$$\epi u_{\delta,\rho}^{b} = \epi \l_{T_\delta} \cap \{x_1 \leq b(1+\delta)+\rho(x_{n+1}-b)\}.$$
That is
\begin{equation}
\label{eq:def_ub_lvl_sets}
\{u_{\delta,\rho}^{b} \leq t\} =\begin{cases}
\emptyset,\qquad &\text{if } t<0\\
t T_\delta,\qquad &\text{if } 0\leq t <b\\
P_{\delta,\rho}^{b,t},\qquad&\text{if }t\geq b.
\end{cases}
\end{equation}
Hence, by Lemma~\ref{le:p_t} we have
$$u_{\delta,\rho}^{b} \mn (\l_{T_\delta}\circ \tau_{b x_\delta}^{-1}+b) = \l_{T_\delta}\quad \text{and} \quad u_{\delta,\rho}^{b} \mx (\l_{T_\delta}\circ \tau_{b x_\delta}^{-1}+b) = \l_{T_\delta \cap \{x_1 \leq \rho \}}\circ \tau_{b x_\delta}^{-1}+b,$$
where $\tau_{b x_\delta}$ denotes the translation $x\mapsto x+b x_\delta$ and where all occurring functions are elements of $\CVf$. Hence, translation invariance together with the valuation property of $\oZ$ shows
$$\oZ(u_{\delta,\rho}^{b}) + \oZ(\l_{T_\delta}+b) = \oZ(\l_{T_\delta}) + \oZ(\l_{T_\delta \cap \{x_1 \leq \rho \}} +b).$$
By Lemma~\ref{le:hd_conv_lvl_sets} and (\ref{eq:def_ub_lvl_sets}) we have $u_{\delta,\rho}^{b}\eto \l_{T_\delta}$ as $b\to +\infty$. By the continuity of $\oZ$ and Lemma~\ref{le:growth_functions} we now have
\begin{align*}
0 &=\lim_{b\to+\infty} \oZ(\l_{T_\delta}) - \oZ(u_{\delta,\rho}^{b})\\
&=\lim_{b\to +\infty} \big(\oZ(\l_{T_\delta}+b) - \oZ(\l_{T_\delta \cap \{x_1 \leq \rho \}} +b) \big)\\
&= \lim_{b\to +\infty} \left(\psi_1(b)\big(V_n(T_\delta) - V_n(T_\delta \cap \{x_1 \leq \rho \})\big) + \psi_2(b)\big(V_n^*(T_\delta) - V_n^*(T_\delta \cap \{x_1 \leq \rho \})\big)\right).
\end{align*}
Repeating the construction above but composing each of the occurring functions with $x\mapsto \frac x2$ and considering that $V_n$ and $V_n^*$ have different degrees of homogeneity, one easily deduces that $\lim_{b\to +\infty}\psi_1(b)= \lim_{b\to +\infty} \psi_2(b)=0$.\par
Assume now that there does not exist $T\in\R$ as claimed. In this case, there exists a sequence $t_k$, $k\in\N$ such that $t_k< t_{k+1}$, $\lim_{k\to+\infty} t_k = +\infty$ and $\psi_2(t_k)>0$ for every $k\in\N$. We repeat the construction above with $b=t_k$ and $\rho=\rho_k:=(1+\psi_2(t_k)^{-1})^{-1}$ for $k\in\N$. Again, we have $u_{\delta,\rho_k}^{t_k}\eto \l_{T_\delta}$ as $k\to+\infty$. Using the continuity of $\oZ$ and the second part of Lemma~\ref{le:p_t} this gives
\begin{align*}
0&=\lim_{k\to +\infty} \oZ(\l_{T_\delta}+t_k)-\oZ(\l_{T_\delta \cap \{x_1 \leq \rho_k\}} + t_k)\\
&= \lim_{k\to +\infty} \psi_1(t_k) (V_n(T_\delta) - V_n(T_\delta\cap \{x_1 \leq \rho_k\}))+\psi_2(t_k)(V_n^*(T_\delta)-V_n^*(T_\delta\cap\{x_1 \leq \rho_k\}))\\
&= 0 -\lim_{k\to +\infty} \psi_2(t_k)(\tfrac{1}{n!\delta^{n-2}} \tfrac{1+\delta}{\delta(1-(n-2)\delta)} (1+\tfrac{1}{\psi_2(t_k)} - \tfrac{1}{1+\delta}))\\
&=\tfrac{1}{n!\delta^{n-2}} \tfrac{1+\delta}{\delta(1-(n-2)\delta)},
\end{align*}
which is a contradiction.
\end{proof}

\begin{lemma}
\label{le:mcmullen_z_reg}
Let $\oY$ be a continuous and translation invariant valuation on $\CVf$. For every $\delta >0$ the map
$$u \mapsto \oY_{\delta}(u):= \oY(\reg_{\delta} u)$$
defines a continuous and translation invariant valuation on $\CV$. Furthermore, for $0\leq i \leq n$, there exist $\varphi_{i,\delta}:\R\times \Kn \to \R$ such that for every $t\in\R$, the map $K\mapsto \varphi_{i,\delta}(t,K)$ is a continuous, translation invariant, $i$-simple valuation that is homogeneous of degree $i$ and
$$\oY_{\delta} (\Ind_K+t) = \sum_{i=0}^n \varphi_{i,\delta}(t,K)$$
for every $K\in\Kn$ and $t\in\R$. Moreover, the maps $\delta\mapsto \varphi_{i,\delta}(t,K)$ and $t\mapsto \varphi_{i,\delta}(t,K)$ are continuous
and if $\oY$ is non-negative, then so are $\oY_{\delta}$ and $\varphi_{i,\delta}$.
\end{lemma}
\begin{proof}
For $u,v\in\CV$ such that $u\wedge v\in\CV$, we have by Lemma~\ref{le:lip_reg} and the valuation property of $\oY$
\begin{align*}
\oY_{\delta}(u\wedge v) + \oY_{\delta}(u\vee v) &= \oY(\reg_{\delta} (u\wedge v)) + \oY(\reg_{\delta} (u\vee v))\\
&= \oY(\reg_{\delta} u \wedge \reg_{\delta} v) + \oY(\reg_{\delta} u \vee \reg_{\delta} v)\\
&= \oY(\reg_{\delta} u) + \oY(\reg_{\delta} v)\\
&= \oY_{\delta}(u)+\oY_{\delta}(v),
\end{align*}
which shows the valuation property of $\oY_{\delta}$. Similar, one shows that $\oY_{\delta}$ is continuous and translation invariant.\par
For $t\in\R$ and $\delta>0$, define $\mu_{\delta,t}:\Kn\to\R$ as
$$\mu_{\delta,t}(K)=\oY_{\delta}(\Ind_K +t).$$
By the properties of $\oY_{\delta}$ it is easy to see that $\mu_{\delta,t}$ defines a continuous and translation invariant valuation on $\Kn$. By Theorem~\ref{thm:mcmullen_decomp} there exist continuous, translation invariant, $i$-simple valuations $(\mu_{\delta,t})_i$ that are homogeneous of degree $i$, $0\leq i \leq n$, such that $\mu_{\delta,t} = (\mu_{\delta,t})_0+\cdots +(\mu_{\delta,t})_n$. Since $t\in\R$ and $\delta>0$ were arbitrary, this defines functions $\varphi_{i,\delta}(t,K)=(\mu_{\delta,t})_i (K)$ for $0\leq i \leq n$. As $t\mapsto \oY_{\delta}(\Ind_K+t)$ is continuous, it is easy to see that the maps $t\mapsto \varphi_{i,\delta}(t,K)$ are continuous as well. Furthermore, by (\ref{eq:def_reg_delta}) and Lemma~\ref{le:hd_conv_lvl_sets} the map $\delta\mapsto \oY_{\delta}(\Ind_K+t)$ is continuous, which together with homogeneity shows that also the maps $\delta \mapsto \varphi_{i,\delta}(t,K)$ are continuous.\par
Finally, if $\oY$ is non-negative, then by definition also $\oY_{\delta}$ is non-negative and consequently the valuations $\varphi_{i,\delta}$ are non-negative as well, which can be seen by evaluating at convex bodies $K$ of different dimensions.
\end{proof}

\begin{lemma}
\label{le:y_only_vol}
Let $\oZ$ be a non-negative, continuous, $\SLn$ and translation invariant valuation on $\CVf$ and let $\psi_0,\psi_1$ and $\psi_2$ denote its growth functions. The map
$$\oY(u)=\oZ(u)-\psi_0(\min\nolimits_{x\in\Rn} u(x)) - \int_{\dom u^*} \psi_2(\nabla u^*(x)\cdot x - u^*(x)) \d x$$
defines a continuous, $\SLn$ and translation invariant valuation on $\CVf$. Furthermore, if $\varphi_{i,\delta}:\R\times\Kn\to\R$ are given as in Lemma~\ref{le:mcmullen_z_reg} such that
$$\oY_{\delta}(\Ind_K+t)=\sum_{i=0}^n \varphi_{i,\delta}(t,K),$$
then $\varphi_{i,\delta}$ is non-negative for every $1\leq i \leq n$ and $\delta>0$. Moreover,
$$\oY_{\delta}(\l_K+t) = \psi_1(t)V_n(K_{\delta})$$
for every $K\in\Ko$ and $t\in\R$, where $K_{\delta}=\conv(K\cup \delta C^n)$.
\end{lemma}
\begin{proof}
By Lemma~\ref{le:min_is_a_val} and Lemma~\ref{le:polar_vol_is_a_val} the operator $\oY$ defines a continuous, $\SLn$ and translation invariant valuation on $\CVf$. By Lemma~\ref{le:growth_functions} we have
\begin{align*}
\oY(\l_L+t)&=\oZ(\l_L+t)-\psi_0(t)-\int_{L^*} \psi_2(t) \d x\\
&= \psi_0(t)+\psi_1(t)V_n(L)+\psi_2(t) V_n^*(L)-\psi_0(t)-\psi_2(t) V_n^*(L)\\
&= \psi_1(t)V_n(L)
\end{align*}
for every $L\in\Koin$. Therefore, by Lemma~\ref{le:lip_reg}
$$\oY_{\delta}(\l_K+t)=\oY(\reg_{\delta} \l_K+t)=\oY(K_{\delta}+t)=\psi_1(t)V_n(K_{\delta}),$$
for every $K\in\Ko$.\par
For arbitrary $t\in\R$ and $\delta>0$ observe that
$$K \mapsto \psi_0(\min\nolimits_{x\in\Rn} \reg_{\delta} \Ind_K(x)+t) = \psi_0(t)$$
is a translation invariant valuation on $\Kn$ that is homogeneous of degree $0$. Furthermore, for any $K\in\Kn$ one has
$$\nabla h(K,x)\cdot x = h(K,x)$$
for a.e. $x\in\Rn$. Setting $u_K=\reg_{\delta} \Ind_K + t$ with $K\in\Kn$ we have
$$u_K^* = (\Ind_K+t)^* + \Ind_{\delta^{-1} Q^n} = h(K,\cdot)-t+\Ind_{\delta^{-1} Q^n}$$
which shows that also
\begin{align*}
K&\mapsto \int_{\dom u_K^*} \psi_2(\nabla u_K^*(x)\cdot x - u_K^*(x)) \d x\\
&= \int_{\delta^{-1} Q^n} \psi_2(\nabla h(K,x)\cdot x - h(K,x)+t) \d x\\
&= \delta^n V_n(Q^n) \psi_2(t),
\end{align*}
is a translation invariant valuation on 
$\Kn$ that is homogeneous of degree $0$. Hence, if we apply Lemma~\ref{le:mcmullen_z_reg} to both $\oY$ and $\oZ$ we obtain
\begin{align*}
\sum_{i=0}^n \varphi_{i,\delta}(t,K) &= \oY_{\delta}(\Ind_K+t)\\
&= \oZ_{\delta}(\Ind_K+t)-\psi_0(t)-\delta^n V_n(Q^n)\psi_2(t)\\
&= \sum_{i=0}^n \rho_{i,\delta}(t,K)-\psi_0(t)-\delta^n V_n(Q^n)\psi_2(t)
\end{align*}
with $\rho_{i,\delta}:\R\times\Kn\to[0,+\infty)$. By homogeneity, we must have $\varphi_{i,\delta}=\rho_{i,\delta}$ for every $1\leq i \leq n$ and $\delta>0$ and in particular, those maps are non-negative.
\end{proof}

By Lemma~\ref{le:val_det_by_gf} every non-negative, continuous, $\SLn$ and translation invariant valuation $\oZ$ on $\CVf$ is uniquely determined by its growth functions $\psi_0,\psi_1,\psi_2:\R\to[0,+\infty)$ and by Lemma~\ref{le:growth_functions_cond} we know that $\lim_{t\to+\infty} \psi_1(t)=0$. In Lemma~\ref{le:diff_lemma} and Lemma~\ref{le:gf_diffable} we will furthermore show that $\psi_1$ is $n$-times continuously differentiable and that its $n$-th derivative has constant sign. The idea of the proof of Lemma~\ref{le:diff_lemma} is to describe the behaviour of $\oZ$ on regularizations of indicator functions with the help of $\psi_1$. As the proof is rather technical, we will describe here its basic idea in the $1$-dimensional case.\par
Let $\oY_{\delta}$ be defined as in Lemma~\ref{le:y_only_vol} and recall that $\oY_{\delta}(\l_K+t)=\psi_1(t) V_n(K_\delta)$ for every $K\in\KoOne$ and $t\in\R$. We will now compute $\oY_{\delta}(\Ind_{[0,\lambda]}+t)$ for $\lambda>0$ and $t\in\R$. By the properties of $\oY_{\delta}$ we have
\begin{align*}
\oY_{\delta}(\l_{[0,\lambda/h]}+t)&=\psi_1(t)V_1\left(\left[0,\tfrac{\lambda}{h}\right]_\delta\right)\\
&= \psi_1(t) V_1\left(\conv\left(\left[0,\tfrac{\lambda}{h}\right]\cup[-\delta,\delta]\right)\right)\\
&= \psi_1(t)\left(\max\left\{\tfrac{\lambda}{h},\delta\right\}+\delta\right)
\end{align*}
for every $\lambda,h>0$ and $t\in\R$. Let $v^{\lambda,h}=\l_{[0,\lambda/h]} + \Ind_{[0,\lambda]}$, that is $v^{\lambda,h}(x)=hx/\lambda$ for $0\leq x \leq \lambda$ and $v^{\lambda,h}(x)=+\infty$ if $x<0$ or $x>\lambda$. Note that $v^{\lambda,h} \eto \Ind_{[0,\lambda]}$ as $h\to 0^+$. In order to calculate $\oY_{\delta}(v^{\lambda,h}+t)$, let $\tau_\lambda$ denote the translation $x\mapsto x+\lambda$ and obverse that
$$v^{\lambda,h} \mn (\l_{[0,\lambda/h]}\circ \tau_\lambda^{-1} + h) = \l_{[0,\lambda/h]}$$
and
$$v^{\lambda,h} \mx (\l_{[0,\lambda/h]}\circ \tau_\lambda^{-1} + h) = \Ind_{\{\lambda\}}+h.$$
Note, that $\Ind_{\{\lambda\}}=\elim_{h\to+\infty} \l_{[0,\lambda/h]}\circ \tau_{\lambda}^{-1}$ and therefore using the continuity as well as the translation invariance of $\oY_{\delta}$ we obtain
\begin{align*}
\oY_{\delta}(\Ind_{\{\lambda\}}+t)&=\lim_{h\to +\infty} \oY_{\delta}(\l_{[0,\lambda/h]}\circ\tau_{\lambda}^{-1}+t)\\
&= \lim_{h\to +\infty} \psi_1(t)\left(\max\left\{\tfrac{\lambda}{h},\delta\right\}+\delta\right)\\
&=2\delta \psi_1(t)
\end{align*}
for every $\delta,\lambda>0$ and $t\in\R$. Thus, we have by the valuation property of $\oY_{\delta}$ together with translation invariance that
\begin{align*}
\oY_{\delta}(v^{\lambda,h}+t) &= \oY_{\delta}(\l_{[0,\lambda/h]}+t)+\oY_{\delta}(\Ind_{\lambda}+h+t) - \oY_{\delta}(\l_{[0,\lambda/h]}\circ \tau_\lambda^{-1} + h+t)\\
&=(\psi_1(t)-\psi_1(t+h))\left(\max\left\{\tfrac{\lambda}{h},\delta\right\}+\delta\right) + 2\delta \psi_1(t+h).
\end{align*}
By continuity, this gives
\begin{align*}
\oY_{\delta}(\Ind_{[0,\lambda]}+t)&=\lim_{h\to 0^+} \oY_{\delta}(v^{\lambda,h}+t)\\
&= \lim_{h\to 0^+} \left( \lambda\tfrac{\psi_1(t)-\psi_1(t+h)}{h} + \delta(\psi_1(t)-\psi_1(t+h)) + 2 \delta \psi_1(t+h) \right)\\
&=\lambda \lim_{h\to 0^+} \tfrac{\psi_1(t)-\psi_1(t+h)}{h} + 2 \delta \psi_1(t),
\end{align*}
which shows that $\psi_1$ is differentiable from the right. Similarly, one has
\begin{align*}
\oY_{\delta}(\Ind_{[0,\lambda]}+t)&=\lim_{h\to 0^+} \oY_{\delta}(v^{\lambda,h}+t-h)\\
&=\lambda \lim_{h\to 0^+} \tfrac{\psi_1(t-h)-\psi_1(t)}{h} + 2 \delta \psi_1(t),
\end{align*}
which shows that $\psi_1$ is differentiable from the left and thus
$$\oY_{\delta}(\Ind_{[0,\lambda]}+t) = -\lambda \psi_1'(t) + 2\delta \psi_1(t),$$
for every $\delta,\lambda > 0$ and $t\in\R$.

\begin{lemma}
\label{le:diff_lemma}
Let $\oY_{\delta}$ be defined as in Lemma~\ref{le:y_only_vol}. For $\lambda>0$,
$$\oY_{\delta}(\Ind_{[0,\lambda]^n}+t) = \frac{1}{n!} \sum_{k=0}^{n} c_{n,k}(\delta) (-\lambda)^k \psi_1^{(k)}(t),$$
for every $t\in\R$, where $c_{n,k}(\delta)$ are polynomials in $\delta$ with $c_{n,n}(\delta)\equiv 1$ and $\psi_1^{(i)}(t)= \frac{\d^i}{\d t^i} \psi_1(t)$ for $i\geq 0$. In particular, $\psi_1$ is $n$-times continuously differentiable.
\end{lemma}
\begin{proof}
Let $\{e_1,\ldots,e_n\}$ denote the standard basis of $\Rn$ and set $e_0=0$. For $h=(h_1,\ldots,h_n)$, $\lambda >0$ and $0\leq i < n$, define the function $u_{i}^{\lambda,h}$ through its sublevel sets as
$$\{u_{i}^{\lambda,h} < 0 \} = \emptyset,\quad \{u_{i}^{\lambda,h}\leq s\} = [0, \lambda e_0]+\cdots [0,\lambda e_i]+\conv\{0,s \lambda e_{i+1}/h_{i+1},\ldots,s \lambda e_n/h_n\},$$
for every $s\geq 0$. Note, that $u_{i}^{\lambda,h}$ does not depend on $h_j$ for $j\leq i$. Furthermore, let $u_n^{\lambda,h} = \Ind_{[0,\lambda]^n}$.
We will use induction to show that $u_i^{\lambda,h}\in\CV$ and
$$\oY_{\delta}(u_i^{\lambda,h}+t) = \frac{1}{n!} \left(\sum_{k=0}^{i} c_{i,k}(\delta) (-\lambda)^k \psi_1^{(k)}(t) \right) \prod_{l=i+1}^n \big(\max\big\{\tfrac{\lambda}{h_l},\delta\big\}+\delta\big),$$
for every $t\in\R$, $\lambda>0$ and $0\leq i\leq n$
with polynomials $c_{i,k}$ such that $c_{i,i}(\delta)\equiv 1$.\par
For $i=0$, let $P_h = \conv\{0, e_1/h_1,\ldots, e_n/h_n\}\in\Po$. Observe, that $u_0^{\lambda,h}=\l_{\lambda P_h}\in\CV$. Hence, by the properties of $\oY_{\delta}$ and Lemma~\ref{le:vol_conv_poly} we have
\begin{align*}\oY_{\delta}(u_0^{\lambda,h}+t) = \oY_{\delta}(\l_{\lambda P_h}+t) &= \psi_1(t)V_n((\lambda P_h)_{\delta}) \\
&= \psi_1(t)V_n(\conv(\delta C^n\cup \lambda P_h))\\
&= \psi_1(t) \frac{1}{n!} \prod_{k=1}^n (\max\{\tfrac{\lambda}{h_k},\delta\}+\delta),
\end{align*}
for every $t\in\R$.\par
Now assume that the statement holds true for $i\geq 0$. Define the function $v_{i+1}^{\lambda,h}$ by
$$\{v_{i+1}^{\lambda,h}\leq s\} = \{u_{i}^{\lambda,h}\leq s\} \cap \{x_{i+1}\leq \lambda\},$$
for every $s\in\R$. As $\epi v_{i+1}^{\lambda,h} = \epi u_{i}^{\lambda,h}\cap \{x_{i+1}\leq \lambda\}$, it is easy to see that $v_{i+1}^{\lambda,h}\in\CV$. Furthermore, by convergence of level sets and Lemma~\ref{le:hd_conv_lvl_sets} we have $v_{i+1}^{\lambda,h}\eto u_{i+1}^{\lambda,h}$ as $h_{i+1}\to 0^+$. Since epi-limits of convex functions are again convex (Theorem~\ref{thm:g_comp}), it is easy to see that $u_{i+1}^h\in\CV$. Now, let $\tau_{i+1}^{\lambda}$ be the translation $x\mapsto x+\lambda e_{i+1}$. Note, that
\begin{equation}
\label{eq:min_u_v}
\{v_{i+1}^{\lambda,h}\leq s\} \cup \{(u_{i}^{\lambda,h}\circ (\tau_{i+1}^\lambda)^{-1}+h_{i+1})\leq s \}=\{u_i^{\lambda,h}\leq s\}.
\end{equation}
Furthermore, let
\begin{equation}
\label{eq:max_u_v}
w_{i+1}^{\lambda,h}:= v_{i+1}^{\lambda,h} \vee (u_{i}^{\lambda,h}\circ (\tau_{i+1}^\lambda)^{-1}+h_{i+1})\in \CV
\end{equation}
and observe that $\dom w_{i+1}^{\lambda,h} \subset \{x_{i+1}=\lambda\}$ as well as
$$\bar{w}_{i+1}^{\lambda,h}:=\elim_{h_{i+1}\to 0^+} w_{i+1}^{\lambda,h} = \elim_{h_{i+1}\to +\infty} u_{i}^{\lambda,h}\circ (\tau_{i+1}^{\lambda})^{-1} \in\CV.$$
Hence, using the induction assumption and the continuity as well as the translation invariance of $\oY_{\delta}$, we obtain
\begin{align*}
\oY_{\delta}(\bar{w}_{i+1}^{\lambda,h}+t) &= \lim_{h_{i+1}\to +\infty} \oY_{\delta}(u_i^{\lambda,h}+t)\\
&= \lim_{h_{i+1}\to +\infty} \frac{1}{n!} \left(\sum_{k=0}^{i} c_{i,k}(\delta) (-\lambda)^k \psi_1^{(k)}(t) \right) \prod_{l=i+1}^n \big(\max\big\{\tfrac{\lambda}{h_l},\delta\big\}+\delta\big)\\
&= \frac{2\delta}{n!} \left(\sum_{k=0}^{i} c_{i,k}(\delta) (-\lambda)^k \psi_1^{(k)}(t) \right) \prod_{l=i+2}^n \big(\max\big\{\tfrac{\lambda}{h_l},\delta\big\}+\delta\big).
\end{align*}
Furthermore, (\ref{eq:min_u_v}) and (\ref{eq:max_u_v}) together with the valuation property of $\oY_{\delta}$ give
$$\oY_{\delta}(u_i^{\lambda,h}+t) + \oY_{\delta}(w_{i+1}^{\lambda,h}+t) = \oY_{\delta}(v_{i+1}^{\lambda,h}+t) + \oY_{\delta}(u_{i}^{\lambda,h}\circ (\tau_{i+1}^\lambda)^{-1}+h_{i+1}+t).$$
Using the induction assumption and the translation invariance of $\oY_{\delta}$ again, we obtain
\begin{align*}
\oY_{\delta}(v_{i+1}^{\lambda,h}+t) &= \oY_{\delta}(u_i^{\lambda,h}+t)-\oY_{\delta}(u_i^{\lambda,h}+t+h_{i+1})+\oY_{\delta}(w_{i+1}^{\lambda,h}+t)\\
&= \frac{1}{n!} \left(\sum_{k=0}^{i} c_{i,k}(\delta) (-\lambda)^k \big(\psi_1^{(k)}(t)-\psi_1^{(k)}(t+h_{i+1})\big) \right) \prod_{l=i+1}^n \big(\max\big\{\tfrac{\lambda}{h_l},\delta\big\}+\delta\big)\\
&\,\quad + \oY_{\delta}(w_{i+1}^{\lambda,h}+t).
\end{align*}
As $h_{i+1}\to 0^+$, the continuity of $\oY_{\delta}$ gives
\begin{align*}
\oY_{\delta}(u_{i+1}^{\lambda,h}+t) &= \lim_{h_{i+1}\to 0^+} \oY_{\delta}(v_{i+1}^{\lambda,h}+t)\\
&=\frac{1}{n!} \left(\sum_{k=0}^{i} c_{i,k}(\delta) (-\lambda)^k \lambda \lim_{h_{i+1}\to 0^+} \frac{\psi_1^{(k)}(t)-\psi_1^{(k)}(t+h_{i+1})}{h_{i+1}} \right)\prod_{l=i+2}^n \big(\max\big\{\tfrac{\lambda}{h_l},\delta\big\}+\delta\big)\\
&\,\quad +\oY_{\delta}(\bar{w}_{i+1}^{\lambda,h}+t),
\end{align*}
which shows that $\psi_1^{(i)}$ is differentiable from the right. Similarly, using $v_{i+1}^{\lambda,h}+t-h_{i+1} \eto u_{i+1}^{\lambda,h} +t$ as $h_{i+1}\to 0^+$, shows that $\psi_1^{(i)}$ is differentiable from the left. Collecting terms therefore shows for every $t\in\R$
\begin{align*}
\oY_{\delta}(u_{i+1}^{\lambda,h}+t)&=\frac{1}{n!} \left(\sum_{k=0}^{i} c_{i,k}(\delta) \big((-\lambda)^{k+1} \psi_1^{(k+1)}(t)+2\delta (-\lambda)^k\psi_1^{(k)}(t) \big)\right)\prod_{l=i+2}^n \big(\max\big\{\tfrac{\lambda}{h_l},\delta\big\}+\delta\big)\\
&=\frac{1}{n!} \left(\sum_{k=0}^{i+1} c_{i+1,k}(\delta) (-\lambda)^k \psi_1^{(k)}(t)\big) \right) \prod_{l=i+2}^n \big(\max\big\{\tfrac{\lambda}{h_l},\delta\big\}+\delta\big)
\end{align*}
with $c_{i+1,1}(\delta)=2\delta$, $c_{i+1,k}(\delta)=c_{i,k-1}(\delta)+2\delta c_{i,k}(\delta)$ for $1<k<i+1$ and $c_{i+1,i+1}(\delta)=c_{i,i}(\delta)\equiv 1$.
\end{proof}

\begin{lemma}
\label{le:gf_diffable}
Let $\oZ$ be a non-negative, continuous, $\SLn$ and translation invariant valuation on $\CVf$. If $\psi_1$ denotes the growth function of $\oZ$, then $\psi_1$ is $n$-times continuously differentiable and $(-1)^n \psi_1^{(n)}$ is non-negative.
\end{lemma}
\begin{proof}
By Lemma~\ref{le:y_only_vol}, we have for $\lambda,\delta>0$ and $t\in\R$.
$$\oY_{\delta}(\Ind_{[0,\lambda]^n}+t) = \sum_{k=0}^{n} \lambda^k\varphi_{k,\delta}(t,[0,1]^n),$$
where $\varphi_{k,\delta}(t,[0,1]^n)$ is non-negative for every $1\leq k\leq n$. Furthermore, by Lemma~\ref{le:diff_lemma}, we have
$$\oY_{\delta}(\Ind_{[0,\lambda]^n}+t) = \frac{1}{n!} \sum_{k=0}^{n} \lambda^k (-1)^k c_{n,k}(\delta) \psi_1^{(k)}(t),$$
with $c_{n,n}(\delta)\equiv 1$. The result now follows by equating coefficients.
\end{proof}

\begin{lemma}
\label{le:moment_cond}
Let $\zeta:\R\to\R$ be $n$-times continuously differentiable. If $\lim_{t\to+\infty}\zeta(t)=0$ and $\zeta^{(n)}$ has constant sign on $[t_0,+\infty)$ for some $t_0\in\R$, then
$$\int_{0}^{+\infty} r^{n-1} \frac{(-1)^n}{(n-1)!} \zeta^{(n)}(r+t) \d r = \zeta(t),$$
for every $t\in\R$. In particular, $\zeta^{(n)}$ has finite moment of order $n-1$. Moreover,
$$\lim_{r\to+\infty} r^k \zeta^{(k)}(r+t)=0$$
for every $0\leq k \leq n-1$ and $t\in\R$.
\end{lemma}
\begin{proof}
Throughout the proof we will assume that $(-1)^n\zeta^{(n)}$ is non-negative on $[t_0,+\infty)$, since we can always consider $-\zeta$ instead of $\zeta$. We use induction on $n$ and start with the case $n=1$. Integration by parts together with the assumption on $\zeta$ gives
\begin{align*}
\int_0^{+\infty} (-1) \zeta'(r+t) \d r &= \lim_{R\to+\infty} \int_0^{R} (-1) \zeta'(r+t)\d r\\
&= \lim_{R\to+\infty} \big( \zeta(t)-\zeta(t+R)\big)\\
&= \zeta(t),
\end{align*}
for every $t\in\R$.\par
Next, let $n\geq 2$ and assume that the statement holds true for $n-1$. Since $(-1)^n\zeta^{(n)}$ is non-negative on $[t_0,+\infty)$, the function $(-1)^{n-1}\zeta^{(n-1)}$ is non-increasing on $[t_0,+\infty)$ and therefore $(-1)^{n-1}\zeta^{(n-1)}$ has constant sign on $[t_1,+\infty)$ for some $t_1\in\R$. Hence, by the induction hypothesis
$$\int_0^{+\infty} r^{n-2} \frac{(-1)^{n-1}}{(n-2)!} \zeta^{(n-1)}(r+t) \d r = \zeta(t)$$
and $\lim_{r\to+\infty} r^k \zeta^{(k)}(r+t)=0$ for every $0\leq k\leq n-2$ and $t\in\R$. In particular, $\zeta^{(n-1)}$ has finite moment of order $n-2$ and therefore $\lim_{t\to+\infty} \zeta^{(n-1)}(t)=0$, which implies that $(-1)^{n-1}\zeta^{(n-1)}(t)\geq 0$ for every $t\geq t_0$.
 Using integration by parts, we obtain
\begin{multline}
\label{eq:int_by_parts}
\int_0^{R} r^{n-1} \frac{(-1)^n}{(n-1)!} \zeta^{(n)}(r+t) \d r\\= R^{n-1} \frac{(-1)^n}{(n-1)!}\zeta^{(n-1)}(R+t) + \int_0^{R} r^{n-2} \frac{(-1)^{n-1}}{(n-2)!} \zeta^{(n-1)}(r+t) \d r,
\end{multline}
for every $t\in\R$ and $R>0$. Since $r^{n-1} \frac{(-1)^n}{(n-1)!} \zeta^{(n)}(r+t)\geq 0$ for $r\geq \max\{0,t_0-t\}$, there exists
$$C(t)=\lim_{R\to+\infty} \int_0^{R} r^{n-1} \frac{(-1)^n}{(n-1)!} \zeta^{(n)}(r+t) \d r \in (-\infty,+\infty],$$
for every $t\in\R$. Hence, (\ref{eq:int_by_parts}) implies that $R^{n-1} \frac{(-1)^n}{(n-1)!}\zeta^{(n-1)}(R+t)$ converges to
\begin{align*}
D(t)&=C(t)-\lim_{R\to +\infty} \int_0^{R} r^{n-2} \frac{(-1)^{n-1}}{(n-2)!} \zeta^{(n-1)}(r+t) \d r\\
&= C(t)-\zeta(t) \in (-\infty,+\infty]
\end{align*}
as $R\to+\infty$. Since $(-1)^{n-1}\zeta^{(n-1)}(R+t)\geq 0$ for every $R\geq t_0-t$ we have
$$R^{n-1} \frac{(-1)^n}{(n-1)!}\zeta^{(n-1)}(R+t)\leq 0$$
for every $R\geq \max\{0,t_0-t\}$ and therefore $D(t)\leq 0$, which is only possible if $C(t)< +\infty$ for every $t\in\R$.
\par
It remains to show that $D(t)\equiv 0$. Assume on the contrary that there exists $\bar{t}\in\R$ such that $(n-1)!D(\bar{t})=2\eta <0$. It follows that
$$r^{n-1} (-1)^n \zeta^{(n-1)}(r+\bar{t}) \leq \eta < 0$$
for every $r\geq r_0$ with $r_0\geq 0$ large enough and therefore
$$(-1)^n \zeta^{(n-1)}(r+\bar{t})\leq \frac{\eta}{r^{n-1}}.$$
Consequently,
\begin{align*}
(-1)^n \zeta^{(n-2)}(r+\bar{t}) &= (-1)^n \zeta^{(n-2)}(r_0+\bar{t}) + \int_{r_0}^r (-1)^n \zeta^{(n-1)}(s+\bar{t}) \d s\\
&\leq (-1)^n \zeta^{(n-2)}(r_0+\bar{t}) + \eta \int_{r_0}^r \frac{1}{s^{n-1}} \d s\\
&= (-1)^n \zeta^{(n-2)}(r_0+\bar{t}) + \eta \begin{cases} \log(r)-\log(r_0),\quad &n=2\\
\frac{1}{n-2}\big(\tfrac{1}{r_0^{n-2}} - \tfrac{1}{r^{n-2}}\big), \quad & n>2.
 \end{cases}
\end{align*}
If $n=2$, this shows that
\begin{align*}
0 &= \lim_{r\to+\infty} \zeta(r+\bar{t})\\
&\leq \zeta(r_0 + \bar{t}) - \eta \log(r_0) + \eta \lim_{r\to+\infty} \log(r)\\
&= -\infty
\end{align*}
which is a contradiction. If $n>2$, we obtain
\begin{align*}
0 &= \lim_{r\to+\infty} (-1)^n r^{n-2} \zeta^{(n-2)}(r+\bar{t})\\
&\leq -\frac{\eta}{n-2} + \lim_{r\to+\infty} r^{n-2} \left((-1)^n \zeta^{(n-2)}(r_0 +\bar{t}) + \frac{\eta}{(n-2) r_0^{n-2}} \right)\\
&=-\frac{\eta}{n-2} + \lim_{r\to+\infty} r^{n-2} \frac{(-1)^n (n-2) r_0^{n-2} \zeta^{(n-2)}(r_0+\bar{t}) + \eta}{(n-2) r_0^{n-2}},
\end{align*}
which goes to $-\infty$ if $r_0$ is large enough.

\end{proof}

\subsection{Proof of the Theorem}
If $\zeta_0,\zeta_1,\zeta_2:\R\to[0,+\infty)$ are continuous functions such that $\zeta_1$ has finite moment of order $n-1$ and $\zeta_2(t)=0$ for all $t\geq T$ with some $T\in\R$, then Lemma~\ref{le:min_is_a_val}, Lemma~\ref{le:vol_is_a_val} and Lemma~\ref{le:polar_vol_is_a_val} show that
$$u\mapsto \zeta_0(\min\nolimits_{x\in\Rn} u(x)) + \int_{\Rn} \zeta_1(u(x)) \d x + \int_{\dom u^*} \zeta_2(\nabla u^*(x)\cdot x - u^*(x)) \d x,$$
defines a non-negative, continuous, $\SLn$ and translation invariant valuation on $\CVf$.\par
Conversely, let $\oZ:\CVf\to[0,+\infty)$ be a continuous, $\SLn$ and translation invariant valuation on $\CVf$. By Lemma~\ref{le:growth_functions} the valuation $\oZ$ has non-negative, continuous growth functions $\psi_0,\psi_1,\psi_2$. Lemma~\ref{le:growth_functions_cond} shows that there exists $T\in\R$ such that $\psi_2(t)=0$ for all $t\geq T$. Furthermore, by Lemma~\ref{le:gf_diffable} the function $\psi_1$ is $n$-times continuously differentiable and $\zeta_1(t):=\frac{(-1)^n}{n!} \psi_1^{(n)}(t)$ is non-negative. Moreover, Lemma~\ref{le:growth_functions_cond} together with Lemma~\ref{le:moment_cond} shows that $\zeta_1$ has finite moment of order $n-1$ and
$$n\int_0^{+\infty} r^{n-1} \zeta_1(r+t)\d r = \psi_1(t),$$
for every $t\in\R$. Finally, for $u=\l_{\lambda B^n} +t$ with $\lambda >0$ and $t\in\R$ we have $u^*=\Ind_{\lambda^{-1} B^n}-t$ and furthermore
\begin{align*}
\oZ(u) &= \psi_0(t) + \psi_1(t) V_n(\lambda B^n) + \psi_2(t) V_n^*(\lambda B^n)\\
&= \psi_0(t) + \lambda^n V_n(B^n)\psi_1(t) + \psi_2(t) V_n(\lambda^{-1} B^n)\\
&=\psi_0(t) + \lambda^n n V_n(B^n) \int_0^{+\infty} r^{n-1} \zeta_1(r+t) \d r + \psi_2(t) V_n(\lambda^{-1} B^n)\\
&= \psi_0(t) + \lambda^n \int_{\Rn} \zeta_1(|x|+t) \d x + \psi_2(t)V_n(\lambda^{-1} B^n)\\
&= \psi_0(t) + \int_{\Rn} \zeta_1\big(\tfrac{|x|}{\lambda}+t\big)\d x + \int_{\lambda^{-1} B^n} \psi_2(t) \d x\\
&=\psi_0(\min\nolimits_{x\in\Rn} u(x)) + \int_{\Rn} \zeta_1(u(x)) \d x + \int_{\dom{u^*}} \psi_2(\nabla u^*(x)\cdot x - u^*(x))\d x.
\end{align*}
By the first part of the proof,
$$u\mapsto \psi_0(\min\nolimits_{x\in\Rn} u(x)) + \int_{\Rn} \zeta_1(u(x)) \d x + \int_{\dom{u^*}} \psi_2(\nabla u^*(x)\cdot x - u^*(x))\d x$$
defines a non-negative, continuous, $\SLn$ and translation invariant valuation on $\CVf$ and by homogeneity it is easy to see that $\psi_0,\psi_1,\psi_2$ are its growth functions. Thus, Lemma~\ref{le:val_det_by_gf} completes the proof of the Theorem.
\hfill\qedsymbol

\section*{Acknowledgments}
The author would like to thank Jin Li for his helpful remarks regarding the proof of Lemma~\ref{le:reduction}.

\footnotesize

\bigskip\bigskip
\parindent 0pt\footnotesize

\parbox[t]{8.5cm}{
Fabian Mussnig\\
Institut f\"ur Diskrete Mathematik und Geometrie\\
Technische Universit\"at Wien\\
Wiedner Hauptstra\ss e 8-10/1046\\
1040 Wien, Austria\\
e-mail: fabian.mussnig@tuwien.ac.at}

\end{document}